\documentclass[11pt]{article}

\usepackage{colortbl,dcolumn,bezier,color,xcolor}
\usepackage{pstricks,pst-node,pst-text,pst-3d}
\usepackage{manfnt}
\usepackage{graphicx}
\usepackage{mathrsfs}
\usepackage{eso-pic}
\usepackage{subfigure}
\usepackage{xcolor}
\usepackage{chemarrow,chemarr,curves}
\usepackage{amsmath}
\usepackage{amsfonts}
\usepackage{amssymb}
\usepackage{multirow}
\usepackage{mathrsfs}
\usepackage{cite}
\usepackage{ascii}
\usepackage{bbding}
\usepackage{pifont}
\usepackage{ifsym}
\usepackage{wasysym}
\usepackage{upgreek}
\usepackage{algorithm}
\usepackage{nonfloat}
\renewcommand{\Box}{\framebox{\rule{0.3em}{0.0em}}}
\newtheorem{theorem}{Theorem}[section]
\newtheorem{lemma}{Lemma}[section]

\newtheorem{example}{Example}[section]
\newtheorem{remark}{Remark}[section]
\newtheorem{definition}{Definition}[section]

\newcommand{\bgeqn}{\begin{eqnarray}}
\newcommand{\edeqn}{\end{eqnarray}}

\newcommand{\bgeq}{\begin{eqnarray*}}
\newcommand{\edeq}{\end{eqnarray*}}
\newcommand{\bgc}{\begin{center}}
\newcommand{\edc}{\end{center}}

\renewcommand{\Box}{\hfill \rule{2.3mm}{2.3mm}}

\makeatletter

\@addtoreset{equation}{section} \makeatother
\newenvironment{proof}{\noindent{\bf Proof. }}{\hfill $\Box$\medskip}

\newrgbcolor{LemonChiffon}{1. 0.98 0.8}
\newrgbcolor{LightBlue}{0.68 0.85 0.9}
\newrgbcolor{DarkBlue}{0.0 0.25 0.8}
\newrgbcolor{Red}{1.0 0.2 0.2}
\newrgbcolor{pink}{1.0 0.4 0.5}
\newrgbcolor{DarkGreen}{0.0 0.6 0.25}

\title{A novel approach for bilevel programs based on Wolfe duality\thanks{This work was supported in part by NSFC (Nos. 12071280, 11901380, 11971220).}
}

\author{Yuwei Li\thanks{\baselineskip 9pt
School of Management, Shanghai University, Shanghai 200444, China. E-mail: yuwei\_li@shu.edu.cn.}, \
   Gui-Hua Lin\thanks{\baselineskip 9pt
Corresponding author. School of Management, Shanghai University, Shanghai 200444, China. E-mail:
guihualin@shu.edu.cn.}, \
    Jin Zhang\thanks{\baselineskip 9pt
Department of Mathematics, SUSTech International Center for Mathematics, Southern University of Science and Technology, Shenzhen 518055, China. E-mail: zhangj9@sustech.edu.cn.}, \
    Xide Zhu\thanks{\baselineskip 9pt
School of Management, Shanghai University, Shanghai 200444, China. E-mail: xidezhu@shu.edu.cn.}
}

\date{January 26, 2022}

\begin{document}
\maketitle

\baselineskip 16pt \vspace{4pt} \noindent{\bf Abstract.} This paper considers a bilevel program, which has many applications in practice. To develop effective numerical algorithms, it is generally necessary to transform the bilevel program into a single-level optimization problem. The most popular approach is to replace the lower-level program by its KKT conditions and then the bilevel program can be reformulated as a mathematical program with equilibrium constraints (MPEC for short). However, since the MPEC does not satisfy the Mangasarian-Fromovitz constraint qualification at any feasible point, the well-developed nonlinear programming theory cannot be applied to MPECs directly. In this paper, we apply the Wolfe duality to show that, under very mild conditions, the bilevel program is equivalent to a new single-level reformulation (WDP for short) in the globally and locally optimal sense. We give an example to show that, unlike the MPEC reformulation, WDP may satisfy the Mangasarian-Fromovitz constraint qualification at its feasible points. We give some properties of the WDP reformulation and the relations between the WDP and MPEC reformulations. We further propose a relaxation method for solving WDP and investigate its limiting behavior. Comprehensive numerical experiments indicate that, although solving WDP directly does not perform very well in our tests, the relaxation method based on the WDP reformulation is quite efficient.

\vspace{4pt}\noindent{\bf Keywords.} Nonlinear programming, Bilevel programming, Wolfe duality, Constraint qualification, Relaxation method, Convergence.

\vspace{4pt}\noindent{\bf 2010 Mathematics Subject Classification.} 90C30, 90C33, 90C46.

%%%%%%%%%%%%%%%%%%%%%%%%%%%%%%%%%%%%%%%%%%%%%%%%%%%%%%%%%%%%%%%%%%%%%%%%%%%%%%%%%%%%%%%%%%%%%%%%%%%%%%%%%%%%%%%%%%%%%%%%%%%%%%%%%%%
\medskip

\baselineskip 16pt
\parskip 4pt

%%%%%%%%%%%%%%%%%%%%%%%%%%%%%%%%
%%%%%%%%%%%%%%%%%%%%%%%%%%%%%%%%
%%%%%%%%%%%%%%%%%%%%%%%%%%%%%%%%
\section{Introduction}\label{intro}

Bilevel programs are hierarchical optimization problems which contain some parameterized optimization or equilibrium problems as constraints. This type of leader-follower game problems can be traced back to the model first introduced by Stackelberg \cite{stk1934}
and they have been applied to many fields such as game theory, industrial engineering, economic planning, management science, and so on \cite{Bennett,Bard1998practical,%Dempe2002foundations,
Franceschi2018bilevel,Liu2020generic}.
For more details about their developments, we refer the readers to \cite{Dempe2013bilevel,
Bard1998practical,Colson2007overview}
and the references therein.

In this paper, we consider the optimistic bilevel program
\begin{eqnarray*}
(\mathrm{BP})\quad \min &&F(x, y)\\
\mbox{s.t.}&&x\in X,~y \in {S}(x).
\end{eqnarray*}
Here, $F : \mathbb{R}^{n+m}\rightarrow\mathbb{R}$ denotes the leader's objective function, $X\subseteq \mathbb{R}^n$ denotes the leader's strategy set, and ${S}(x) \subseteq \mathbb{R}^m$ denotes the follower's optimal reaction set, which is  the optimal solution set of the parameterized program
\begin{eqnarray*}
(\mathrm{P}_x)\quad \min\limits_{y}&&f(x, y)\\
\mbox{s.t.}&&g(x, y) \le 0, ~h(x, y)=0,
\end{eqnarray*}
where $f: \mathbb{R}^{n+m} \rightarrow\mathbb{R}$, $g$: $\mathbb{R}^{n+m} \rightarrow\mathbb{R}^{p}$, and $h$: $\mathbb{R}^{n+m} \rightarrow\mathbb{R}^{q}$. Throughout the paper, we denote by $Y(x) := \{y \in \mathbb{R}^{m} : g(x, y) \le 0, ~h(x, y) = 0\}$ the feasible region of $({\rm P}_x)$ and assume that $F$ is continuously differentiable, $\{f, g, h\}$ are all twice continuously differentiable, and ${S}(x) \neq \emptyset$ for each $x\in X$. (BP) describes the case where the leader can anticipate how the follower makes its decisions. In practice, there are other cases where the leader and the follower are not cooperative and, in these cases, the leader may make decisions in the pessimistic manner.
See \cite{%Dempe2019solution,
Bennett,Dempe2013bilevel,
Loridan1996weak,%Dempe2002foundations,
Bard1998practical,Lampariello2020numerically} for more details about optimistic and pessimistic bilevel models.

As is well-known to us, due to the hierarchical structure, bilevel programs are very difficult to solve.
Actually, even for linear cases, bilevel programs have been shown to be strongly NP-hard. From the perspective of continuous optimization, to develop effective numerical algorithms, it is generally necessary to transform bilevel programs into single-level optimization problems. Along this way, there have been presented several reformulations for bilevel programs, which include
\begin{itemize}
\item[$\diamond$] the reformulation based on lower-level solution functions,
\item[$\diamond$] the reformulation based on lower-level optimal value functions,
\item[$\diamond$] the reformulation based on lower-level optimality conditions,
\item[$\diamond$] the reformulation based on lower-level Lagrangian duality.
\end{itemize}
See Section 2 given below for details about these reformulations.

Among the reformulations for bilevel programs mentioned above, the first one requires the lower-level solution functions to be known and, in such cases, it is generally neither differentiable nor convex even if both the upper-level and lower-level programs are differentiable and convex. For the second reformulation, the lower-level optimal value functions are usually difficult to evaluate because they do not have analytic expressions in general. The third reformulation is the most popular in  literature and it is known as the mathematical program with equilibrium constraints (MPEC). However, MPECs are highly nonconvex with combinatorial structure and, in theory, MPECs do not satisfy the Mangasarian-Fromovitz constraint qualification (MFCQ) at any feasible point so that the well-developed nonlinear programming theory cannot be applied to MPECs directly.
The fourth reformulation uses the so-called regularized constrained Lagrangian dual function,
but it has some flaws as follows:
\begin{itemize}
\item The regularized constrained Lagrangian dual function used in the duality-based reformulation is defined by minimizing the lower-level Lagrangian function over a compact and convex constraint set and hence it generally has no analytic expressions.

\item The lower-level feasible region is assumed to be uniformly contained within the interior of some compact and convex set, which is too stringent. 

\item If the lower-level objective function is convex but not strictly convex, an approximation parameter is introduced and required to tend to 0 in solving the original bilevel program. 
\end{itemize}

In this paper, we present a new single-level reformulation for (BP) with better structure and properties than the existing ones. Our strategy is to use the Wolfe duality for constrained optimization problems to deal with the lower-level program. Compared with the existing reformulations introduced above, the new reformulation based on lower-level Wolfe duality does not involve any function without analytic expression and, more importantly, it may satisfy the MFCQ at its feasible points.

The paper is organized as follows. In Section 2, we first recall the existing single-level reformulations for (BP) and then introduce some Wolfe duality theorems for constrained optimization problems. In Section 3, we apply the Wolfe duality theorems to present our new single-level reformulation for (BP). We show their equivalence from both globally and locally optimal senses. In Section 4, we make a comparison between the new reformulation and the most popular MPEC reformulation. In particular, we first give an example to show that the new reformulation may satisfy the MFCQ at its feasible points and then investigate the relations between the two reformulations. In Section 5, we present a relaxation scheme for solving the new reformulation and establish a convergence analysis. In Section 6, we report our numerical experiments on a number of linear bilevel programs generated randomly. In Section 7, we make some concluding remarks.

Throughout, for a given differentiable function $G: \mathbb{R}^{n}\rightarrow\mathbb{R}^{m}$, we use $\nabla G(x)\in\mathbb{R}^{n\times m}$ to denote the transposed Jacobian of $G$ at $x$. When $m=1$, $\nabla G(x)\in \mathbb{R}^{n}$ stands for the gradient of $G$ at $x$. Given two vectors $a,b\in \mathbb{R}^{n}$, $a\bot b$ means $a^Tb=0$.

%%%%%%%%%%%%%%%%%%%%%%%%%%%%%%%%
%%%%%%%%%%%%%%%%%%%%%%%%%%%%%%%%
%%%%%%%%%%%%%%%%%%%%%%%%%%%%%%%%
\section{Preliminaries}

In this section, we introduce some existing reformulations for bilevel programs and the Wolfe duality for constrained optimization problems.

%%%%%%%%%%%%%%%%%%%%%%%%%%%%%%%%
%%%%%%%%%%%%%%%%%%%%%%%%%%%%%%%%
%%%%%%%%%%%%%%%%%%%%%%%%%%%%%%%%
\subsection{Existing reformulations for bilevel programs}\label{approaches to BP}

As mentioned before, in order to develop effective numerical algorithms, we need to transform bilevel programs into single-level optimization problems. In this section, we briefly recall the existing reformulations for bilevel programs.

%%%%%%%%%%%%%%%%%%%%%%%%%%%%%%%%
%%%%%%%%%%%%%%%%%%%%%%%%%%%%%%%%
%%%%%%%%%%%%%%%%%%%%%%%%%%%%%%%%
\subsubsection{Reformulation based on lower-level solution functions}

If the lower-level program $({\rm P}_x)$ has a unique optimal solution $y(x)$ for any given $x\in X$,  the bilevel program (BP) can be reformulated equivalently into the single-level problem
\begin{eqnarray*}
(\mathrm{SP})\quad \min &&F(x, y(x)) \\
\mbox{\rm s.t.}&&x\in X.
\end{eqnarray*}

Naturally, this reformulation requires stringent conditions to ensure the existence of lower-level optimal solutions and, under the condition that the solution function $y(x)$ is known, the objective function in the above problem is generally neither differentiable nor convex, even if both the upper-level and lower-level programs are differentiable and convex. Therefore, this approach is usually hard to be applied in solving bilevel programs.

%%%%%%%%%%%%%%%%%%%%%%%%%%%%%%%%
%%%%%%%%%%%%%%%%%%%%%%%%%%%%%%%%
%%%%%%%%%%%%%%%%%%%%%%%%%%%%%%%%
\subsubsection{Reformulation based on lower-level optimal value functions}

Another approach for bilevel programs is to use the lower-level optimal value function defined by
$V(x):=\min\limits_{y \in Y(x)}~f(x, y).$
Then, the bilevel program (BP) can be reformulated equivalently into the single-level optimization problem
\begin{eqnarray*}
(\mathrm{VP})\quad \min &&F(x, y)\\
\mbox{s.t.}&&f(x, y) - V(x) \leq 0, ~x\in X,~y\in Y(x).
\end{eqnarray*}
This reformulation was first proposed by Outrata in \cite{Outrata1990numerical}. Subsequently, Ye et al. \cite{Ye1995optimality,Ye2010new} and Dempe et al. \cite{%Dempe2007new,Dempe2011generalized,
Dempe2013bilevel}
used the lower-level optimal value functions to study various optimality conditions and constraint qualifications
for bilevel programs.

On the other hand, although the reformulation (VP) looks like a single-level optimization problem, as the lower-level optimal value function $V(x)$ does not have analytic expression and may not be differentiable in general, it is obviously not a normal constrained optimization problem. Moreover, it is easy to see that the inequality constraint in (VP) is actually an equality constraint and hence the nonsmooth MFCQ is never satisfied at each feasible point \cite[Proposition 3.2]{Ye1995optimality}. In consequence, we cannot apply the well-developed optimization algorithms to solve it directly. Recently, motivated by the study of semi-infinite programming, Xu et al. \cite{Lin2014solving, Xu2014smoothing} considered a simple bilevel program with a compact lower-level feasible region and proposed some approximation methods by using an integral entropy function to approximate the lower-level optimal value function.

%%%%%%%%%%%%%%%%%%%%%%%%%%%%%%%%
%%%%%%%%%%%%%%%%%%%%%%%%%%%%%%%%
%%%%%%%%%%%%%%%%%%%%%%%%%%%%%%%%
\subsubsection{Reformulation based on lower-level optimality conditions}

Because of the existence of the lower-level optimal solution or optimal value functions, both reformulations introduced above are hard to be applied in solving bilevel programs. The most popular approach so far is to use the lower-level optimality conditions to transform the bilevel program (BP) into an MPEC. If the lower-level feasible region $Y(x)$ is convex for each $x\in X$, by means of the classical first-order necessary conditions for $({\rm P}_x)$, the bilevel program (BP) can be reformulated as a mathematical program with parametric variational inequality constraints, which is essentially a semi-infinite programming problem and mainly useful in the theoretical study. To develop numerical algorithms, we may use the KKT conditions for $({\rm P}_x)$, which usually requires some constraint qualifications, to reformulate the bilevel program (BP) as the following mathematical program with complementarity constraints:
\begin{eqnarray*}
(\mathrm{MPEC})\quad\min  && F(x,y) \\
\mbox{s.t.\!\! } && x \in X,~h(x, y)=0,\\
&&0\leq u ~ \bot ~g(x, y)\le 0,\\
&&\nabla_y f(x, y) + \nabla_y g(x, y) u+ \nabla_y h(x, y)v=0.
\end{eqnarray*}

On the one hand, although the globally optimal solutions of the original bilevel program (BP) and the corresponding MPEC coincide with each other, their local optimal solutions may not be correspond with each other in the case that multiple multipliers exist for the lower-level program \cite{dempe2012bilevel}. In addition, Mirrlees gave an example to reveal that the optimal solutions of the original bilevel program may not even be a stationary point of the corresponding MPEC \cite{Mirrlees1999theory}.

On the other hand, MPEC has a combinatorial structure, that is, its feasible region is usually a union of lots of pieces, and it fails to satisfy the MFCQ at any feasible point, which means that the well-developed optimization algorithms in nonlinear programming may not be stable in solving MPEC. During the past decades, there have been proposed a number of approximation algorithms for solving MPEC, including relaxation and smoothing algorithms, penalty function algorithms, interior point algorithms, implicit programming algorithms, active-set identification algorithms, constrained equation algorithms, nonsmooth algorithms, etc. We refer the readers to
\cite{Fuku-1,Solodov-1,Scholtes2001,%Dempe2019solution,
Fukushima2004smoothing,Hoheisel2013theoretical,LinEquiation,LinSMPEC,LinHybrid,Luo1996mathematical,Pang1999complementarity,Scheel2000mathematical} and the references therein for more details about the recent developments on theory and numerical algorithms for MPEC.

%%%%%%%%%%%%%%%%%%%%%%%%%%%%%%%%
%%%%%%%%%%%%%%%%%%%%%%%%%%%%%%%%
%%%%%%%%%%%%%%%%%%%%%%%%%%%%%%%%
\subsubsection{Reformulation based on lower-level Lagrangian duality}

Recently, Ouattara and Aswani \cite{Ouattara2018duality} studied another approach to deal with bilevel programs by using some kind of Lagrangian dual functions of the lower-level program {($\mathrm{P}_{x}$)}. 
Since the classical Lagrangian dual function of {($\mathrm{P}_{x}$)} may not be differentiable and even not real-valued, under the condition that the lower-level feasible region is convex and uniformly contained within the interior of some compact and convex set $\bar{Y}$, the authors defined a so-called constrained Lagrangian dual function by
\begin{eqnarray*}
\mathcal{V}(x,u,v):=\min_{y\in \bar{Y}} \, f(x,y)+u^T g(x,y)+v^T h(x,y)
\end{eqnarray*}
and showed that, if $f(x,y)$ is strictly convex in $y$ for each $x\in X$, the constrained Lagrangian dual function is differentiable when $u\ge 0$. In this case, they presented the following duality-based reformulation:
\begin{eqnarray*}
(\mathrm{LDP})\quad \min  && F(x,y)  \nonumber\\
\mbox{ s.t.\!\! } && f(x,y)\le\mathcal{V}(x,u,v),\label{RDF+}\\
&&u\geq0,~g(x,y)\le0,\nonumber\\
&&x\in X,~h(x,y)=0.\nonumber
\end{eqnarray*}
If $f(x,y)$ is only convex but not necessarily strictly convex in $y$ for each $x\in X$, they introduced a regularized constrained Lagrangian dual function
\begin{eqnarray*}
\mathcal{V}_\delta(x,u,v):=\min_{y\in \bar{Y}} \, \delta \|y\|^2+f(x,y)+u^T g(x,y)+v^T h(x,y),
\end{eqnarray*}
where $\delta>0$, and gave the following approximated duality-based reformulation:
\begin{eqnarray}
\min && F(x,y)  \nonumber\\
\mbox{s.t.\!\! } && f(x,y)\le\mathcal{V}_\delta(x,u,v),\label{RDF}\\
&&u\geq0,~g(x,y)\le0,\nonumber\\
&&x\in X,~h(x,y)=0.\nonumber
\end{eqnarray}
Note that both $\mathcal{V}$ and $\mathcal{V}_\delta$ do not have analytic expressions in general and, in addition, the approximation parameter $\delta$ in \eqref{RDF} is required to tend to $0$ decreasingly in solving the original bilevel program.

%%%%%%%%%%%%%%%%%%%%%%%%%%%%%%%%
%%%%%%%%%%%%%%%%%%%%%%%%%%%%%%%%
%%%%%%%%%%%%%%%%%%%%%%%%%%%%%%%%
\subsection{Wolfe duality for constrained optimization problems}
\label{Wolfe-duality}

This subsection is devoted to gathering some useful results related to the Wolfe duality for constrained optimization problems, which play a crucial role in our proposed approach to bilevel programs. For simplicity, in this subsection, we omit the upper-level variable $x$ and rewrite the lower-level program ($\mathrm{P}_{x}$) as
\begin{eqnarray*}
(\mathrm{P})\quad \min && f(y)\\
\mbox{s.t.}&& g(y) \le 0, ~h(y) = 0,
\end{eqnarray*}
where $f: \mathbb{R}^{m}\rightarrow\mathbb{R}$, $g$: $\mathbb{R}^{m}\rightarrow\mathbb{R}^{p}$, and $h$: $\mathbb{R}^{m}\rightarrow\mathbb{R}^{q}$. Denote by $Y$ and ${S}$ the feasible region and solution set of (P) respectively.
The Wolfe dual problem given in \cite{Wolfe1961duality} for (P) is
\begin{eqnarray*}
(\mathrm{WD})\quad\max &&L(z, u, v)\\
\mbox{s.t.}&& \nabla_z  L(z, u, v)=0,~u\geq 0,
\end{eqnarray*}
where $L: \mathbb{R}^{m+p+q} \to \mathbb{R}$ represents the Lagrangian function
\begin{eqnarray}\label{Lagrangian}
L(z, u, v) := f(z) + u^{T} g(z) + v^{T} h(z).
\end{eqnarray}
We denote by $W$ the feasible region of {(\rm WD)} and, for any fixed $(u,v)$, we let $Z(u,v):=\{z\in \mathbb{R}^m: \nabla_z L(z, u, v)=0\}$.
Note that, since the Lagrangian function $L$ may not be convex over $W$ due to the existence of the terms $u^{T} g$ and $v^{T} h$, the Wolfe dual problem (WD) is usually a nonconvex optimization problem, even though the primal problem (P) is convex.

\begin{definition}\rm{}\label{pseudoconvex}
We say $f$ to be pseudoconvex on $\mathcal{Y}\subseteq\mathbb{R}^m$ if $(y_1-y_2)^T \nabla f(y_2) \geq 0$ implies $f(y_1) \geq f(y_2)$ for any $y_1, y_2\in \mathcal{Y}$.
\end{definition}

From Theorems 3 and 4 in \cite{Hanson1982further}, we have the following duality theorems.

\begin{theorem}\label{Weak duality0}\emph{\textbf{(weak Wolfe duality)}}
Suppose that $L(\cdot, u, v)$ is pseudoconvex on $Y\cup Z(u,v)$ for any $u\in \mathbb{R}^p_+$ and $v\in \mathbb{R}^q$. Then,
the optimal value of the primal problem {\rm (P)} is no less than its Wolfe dual problem {\rm (WD)}, i.e.,
\begin{eqnarray*}
\min\limits_{y\in Y} f( y)  \geq \max\limits_{(z, u, v)\in {W}} L( z, u, v).
\end{eqnarray*}
\end{theorem}

\begin{proof}
For any $y\in Y$ and $(z,u,v)\in W$, we have
$\nabla_z L(z, u, v)=0$ and hence $L(y, u, v)\geq L(z, u, v)$ by the pseudoconvexity of $L(\cdot, u, v)$.
Since $u\in \mathbb{R}^p_+$ and $y\in Y$ satisfies $h(y)=0$ and $g(y)\leq 0$, we have
\begin{eqnarray*}
 f(y)\geq f(y) + u^{T} g(y) + v^{T} h(y)= L(y, u, v)\geq L(z, u, v).
\end{eqnarray*}
By the arbitrariness of $y\in Y$ and $(z,u,v)\in W$, we get the conclusion.
\end{proof}

\begin{theorem}\label{Strong duality0}\emph{\textbf{(strong Wolfe duality)}}
Suppose that  $L(\cdot, u, v)$ is pseudoconvex on $Y\cup Z(u,v)$ for any $u\in \mathbb{R}^p_+$, $v\in \mathbb{R}^q$
and {\rm (P)} satisfies the Guignard constraint qualification at some solution $\bar{y} \in {S}$. Then, there exists an optimal solution $(\bar{z},\bar{u},\bar{v})\in{W}$ to {\rm (WD)} such that
\begin{eqnarray*}
\min\limits_{y\in Y} f(y) = f(\bar{y}) = L(\bar{z}, \bar{u}, \bar{v}) = \max\limits_{(z, u, v)\in {W}} L(z, u, v).
\end{eqnarray*}
\end{theorem}

\begin{proof}
Since (P) satisfies the Guignard constraint qualification at $\bar{y} \in {S}$,
there exists $(\bar{u},\bar{v})$ such that
\begin{eqnarray*}
\nabla_y L(\bar{y},\bar{u},\bar{v})=0,~~ \bar{u}^T g(\bar{y})=0, ~~\bar{u}\geq 0,
\end{eqnarray*}
from which we have $(\bar{y},\bar{u},\bar{v})\in W$. We can show that $(\bar{y},\bar{u},\bar{v})$ is an optimal solution to $(\mathrm{WD})$. Otherwise,
there must exist some $(y,u,v)\in W$ such that
\begin{eqnarray*}
L(y,u,v)>L(\bar{y},\bar{u},\bar{v})= f(\bar{y}),
\end{eqnarray*}
which contradicts the weak Wolfe duality shown in Theorem \ref{Weak duality0}.
Therefore, $(\bar{y},\bar{u},\bar{v})$ is an optimal solution to $(\mathrm{WD})$.
As a result, we have
\begin{eqnarray*}
\min\limits_{y\in Y} f(y) = f(\bar{y}) = L(\bar{y}, \bar{u}, \bar{v})
= \max\limits_{(z, u, v)\in {W}} L(z, u, v).
\end{eqnarray*}
This completes the proof.
\end{proof}

\begin{remark}\label{weaker convexity}\rm
Note that the pseudoconvexity assumption for the Lagrangian function in Theorems \ref{Weak duality0} and \ref{Strong duality0} can be replaced by some other weaker conditions such as the $\mathcal F$-pseudoconvexity or the invexity. Since these conditions are not easy to understand by definitions, we use the well-known pseudoconvexity to describe our results. Actually, most results can be weakened by replacing the pseudoconvexity assumption with these weaker conditions. Note that
    \begin{center}
Pseudoconvexity ~$\Rightarrow $~ Invexity $\Leftrightarrow$ $\mathcal{F}$-convexity
$\Rightarrow$ $\mathcal{F}$-pseudoconvexity,
\end{center}
where the first implication follows from \cite{Ben-Israel1986what}, the equivalence follows from \cite{Osuna1998invex}, and the last implication follows from \cite{Hanson1982further}. That is, the $\mathcal F$-pseudoconvexity is the weakest condition mentioned above.
We refer the readers to \cite{Hanson1982further,Osuna1998invex,Weir1988pre,Wolfe1961duality} for details about these conditions.
\end{remark}

\begin{remark}\rm
The results under the $\mathcal F$-pseudoconvexity assumption were shown in \cite{Hanson1982further}, where the $\mathcal F$-pseudoconvexity was assumed over the feasible region $Y$ only, rather than the set $Y\cup Z(u,v)$ assumed in Theorems \ref{Weak duality0} and \ref{Strong duality0}. However, it can be seen from their proofs that the $\mathcal F$-pseudoconvexity only over the feasible region  is not sufficient to derive their conclusions. This is why we retain the proofs for these theorems.
\end{remark}

\begin{remark}\rm
It should be noted that, in Theorem \ref{Strong duality0}, the dual optimal solution $\bar{z}$ may not solve the primal problem {\rm(P)}, unless there are some restrictions on the functions involved. See, e.g., the following converse duality theorem given in \cite{Craven1971converse,Kanniappan1984duality}.
\end{remark}

\begin{theorem}\label{Converse duality1}
Suppose that {\rm (P)} is a convex program, 
$\bar{y}$ solves {\rm(P)}, and $(\bar{z}, \bar{u}, \bar{v})$ solves {\rm(WD)}.
Then, we have $\bar{z} = \bar{y}$ and $f(\bar{y}) = L(\bar{z}, \bar{u}, \bar{v})$ if one of the following conditions holds:
\begin{itemize}
  \item[{\rm (i)}] The Hessian matrix $\nabla_{zz}^2 {L}(\bar{z}, \bar{u}, \bar{v})$ is nonsingular.

  \item[{\rm (ii)}] The Guignard constraint qualification holds at $\bar{y}$ and the function $f$ is strictly convex near $\bar{z}$.
\end{itemize}
\end{theorem}

%%%%%%%%%%%%%%%%%%%%%%%%%%%%%%%%
%%%%%%%%%%%%%%%%%%%%%%%%%%%%%%%%
%%%%%%%%%%%%%%%%%%%%%%%%%%%%%%%%
\section{New reformulation for bilevel programs}
\label{WDP to BP}

In this section, by applying the Wolfe duality to the lower-level program, we present a novel single-level reformulation for the bilevel program (BP). Different from the existing reformulations mentioned in Subsection 2.1, the new reformulation does not involve any function without analytic expression and, in addition, an example is given in the next section to show that the new reformulation may satisfy the MFCQ at its feasible points.

As introduced in Subsection \ref{Wolfe-duality}, for any given $x\in X$, the Wolfe dual problem for the lower-level program ($\mathrm{P}_{x}$) is
\begin{eqnarray*}
(\mathrm{WD}_x)\quad\max\limits_{z, u, v}&&L(x, z, u, v)\\
\mbox{s.t.}&& \nabla_z L(x, z, u, v)=0,~u\geq 0,
\end{eqnarray*}
where $L(x, z, u, v) := f(x, z) + u^{T} g(x, z) + v^{T} h(x, z).$
Similarly, for any fixed $(x, u,v)$, we let $Z(x,u,v):=\{z\in \mathbb{R}^m: \nabla_z L(x,z, u, v)=0\}$ and denote by ${W}(x)$ the feasible region of ($\mathrm{WD}_{x}$).

\begin{theorem}\label{WDP}
Suppose that, for any $x\in X$, $L(x,\cdot, u, v)$ is pseudoconvex on $Y(x)\cup Z(x,u,v)$ for any $u\in \mathbb{R}^p_+$, $v\in \mathbb{R}^q$, and $(\mathrm{P}_{x})$ satisfies the Guignard constraint qualification at some solution $y_x \in {S(x)}$.
Then, if $(\bar{x}, \bar{y})$ is an optimal solution to the bilevel program {\rm (BP)}, there exists $(\bar{z}, {\bar{u}}, {\bar{v}})\in {W}(\bar{x})$ such that $(\bar{x}, \bar{y}, \bar{z}, {\bar{u}}, {\bar{v}})$  is an optimal solution to the single-level optimization problem
\begin{eqnarray*}
(\mathrm{WDP})\quad\min  && F(x,y) \\
\mbox{\rm s.t.\!\! } &&x\in X, \ y\in Y(x),\\
&&f(x, y) - L(x, z, u, v) \leq 0,\\
&&\nabla_z L(x, z, u, v) = 0,\ u\geq 0,
\end{eqnarray*}
or equivalently,
\begin{eqnarray*}
\min &&F(x, y)  \\
\mbox{\rm s.t.\!\! } && x\in X,~g(x, y) \le 0, ~h(x, y)=0,\\
&&f(x, y) - f(x, z) - u^{T} g(x, z) - v^{T} h(x, z) \leq 0, \\
&&\nabla_z f(x, z) + \nabla_z g(x, z) u + \nabla_z h(x, z) v = 0, ~u\geq 0.
\end{eqnarray*}
Conversely, if $(\bar{x}, \bar{y}, \bar{z}, {\bar{u}}, {\bar{v}})$ is an optimal solution to {\rm(WDP)}, $(\bar{x}, \bar{y})$ is an optimal solution to {\rm(BP)}.
\end{theorem}

\begin{proof}
Recall that (BP) is equivalent to the value reformulation (VP) with the same optimal solution set.
By the assumptions and Theorem \ref{Strong duality0}, for each $x\in X$,
there exists an optimal solution $(z_x,u_x,v_x)\in {W}(x)$ to {\rm($\mathrm{WD}_{x}$)} such that
$$\min\limits_{y\in Y(x)}f(x,y) = L(x,z_x,u_x,v_x)=\max\limits_{(z, u, v) \in {W}(x)} L(x, z, u, v).$$
Thus, the value reformulation (VP) can be rewritten as
\begin{eqnarray}\label{wdvbp}
\min && F(x,y) \nonumber\\
\mbox{s.t.} &&x\in X, \ y\in Y(x),\\
&&f(x, y) \leq \max\limits_{(z, u, v) \in W(x)} L(x, z, u, v).\nonumber
\end{eqnarray}
Next, we show the equivalence between \eqref{wdvbp} and (WDP).

(i) Suppose that $(\bar{x}, \bar{y})$ is an optimal solution to \eqref{wdvbp}. Then, we have
\begin{eqnarray}\label{th4-00}
F(\bar{x}, \bar{y}) \leq F(x,y)
\end{eqnarray}
for any feasible point $(x,y)$ to \eqref{wdvbp} and, by the feasibility of $(\bar{x}, \bar{y})$ to \eqref{wdvbp},
\begin{eqnarray*}\label{th4-1}
f(\bar{x}, \bar{y}) \leq \max\limits_{(z, u, v)\in W(\bar{x})} L(\bar{x}, z, u, v)=L(\bar{x}, z_{\bar{x}}, u_{\bar{x}}, v_{\bar{x}}).
\end{eqnarray*}
This together with $(z_{\bar{x}}, u_{\bar{x}}, v_{\bar{x}})\in {W}(\bar{x})$ indicates that $(\bar{x}, \bar{y}, z_{\bar{x}}, u_{\bar{x}}, v_{\bar{x}})$ is feasible to (WDP).
For any feasible point $(x, y, z, u, v)$ to (WDP), since
\begin{eqnarray*}\label{th4-3}
f(x, y) \leq L(x, z, u, v)\leq\max\limits_{(z', u', v') \in {W}(x)} L(x, z', u', v'),
\end{eqnarray*}
where the second inequality follows from $(z, u, v) \in{W}(x)$, $(x, y)$ is feasible to \eqref{wdvbp} and hence we have \eqref{th4-00}. By the arbitrariness of $(x, y, z, u, v)$,  $(\bar{x}, \bar{y}, z_{\bar{x}}, u_{\bar{x}}, v_{\bar{x}})$ is an optimal solution to (WDP).

(ii) Suppose that $(\bar{x}, \bar{y}, \bar{z}, {\bar{u}}, {\bar{v}})$ is an optimal solution to (WDP). We have $\nabla_z L(\bar{x}, \bar{z}, {\bar{u}}, {\bar{v}})=0, \bar{u}\ge 0$ (i.e., $(\bar{z}, {\bar{u}}, {\bar{v}})\in W(\bar{x})$), and
\begin{eqnarray*}
f(\bar{x}, \bar{y}) \leq L(\bar{x},\bar{z}, {\bar{u}}, {\bar{v}})\le \max\limits_{(z, u, v)\in W(\bar{x})} L(\bar{x}, z, u, v).
\end{eqnarray*}
This means that $(\bar{x}, \bar{y})$ is feasible to \eqref{wdvbp}. For any feasible point $(x, y)$ to \eqref{wdvbp}, we have
\begin{eqnarray*}\label{th4-3}
f(x, y) \leq \max\limits_{(z, u, v) \in {W}(x)} L(x, z, u, v)= L(x,z_x,u_x,v_x)
\end{eqnarray*}
and so $(x,y,z_x,u_x,v_x)$ is feasible to (WDP). By the optimality of $(\bar{x}, \bar{y}, \bar{z}, {\bar{u}}, {\bar{v}})$ to (WDP), we have \eqref{th4-00}. By the arbitrariness of $(x, y)$, $(\bar{x}, \bar{y})$ is an optimal solution to \eqref{wdvbp}.

This completes the proof.
\end{proof}

\begin{remark}\label{wdpcondition}\rm
As stated in Remark \ref{weaker convexity}, the pseudoconvexity assumption in  Theorem \ref{WDP} can be weakened by the $\mathcal F$-pseudoconvexity. The following example shows that, even if the lower-level program does not satisfy the $\mathcal F$-pseudoconvexity, the original bilevel program (BP) is still possibly equivalent to (WDP).
\end{remark}

\begin{example}\label{nonconvexsBP=WDP}\rm
Consider the bilevel program
\begin{eqnarray}\label{ex-BP=WDP-1}
\min && 2x-y \\
\mbox{s.t.}&& x\geq0, ~y\in S(x)=\arg\min\limits_y\{-e^{-(y-x)^2}+0.3y^2-0.6xy: y\geq x\}.\nonumber
\end{eqnarray}
For any $x\ge 0$, since the lower-level objective function is strictly increasing with respect to $y$ on $[x,+\infty)$, we always have $S(x)=\{x\}$. Hence, \eqref{ex-BP=WDP-1} has a unique optimal solution $(\bar{x},\bar{y})=(0,0)$ with the optimal value to be $0$.

For this problem, (WDP) becomes
\begin{eqnarray}\label{ex-BP=WDP-2}
\min && 2x-y \nonumber\\
\mbox{s.t.}&& x\geq0, ~y\geq x, \\
&&-e^{-(y-x)^2}+0.3y^2-0.6xy- L(x,z,u)\le 0, \nonumber\\
&&2(z-x)e^{-(z-x)^2}+0.6(z-x)-u=0,~u\geq0,\nonumber
\end{eqnarray}
where the lower-level Lagrangian function $$L(x,z,u):= -e^{-(z-x)^2}+0.3z^2-0.6xz-u(z-x).$$ We next show that, for any given $x\geq 0$, $L(x,\cdot,\bar{u})$ is not $\mathcal F$-pseudoconvex on $Y(x)\cup Z(x,\bar{u})=[x,+\infty)$ for some $\bar{u}\ge 0$, but the original bilevel program \eqref{ex-BP=WDP-1} is still equivalent to \eqref{ex-BP=WDP-2} in the sense of Theorem \ref{WDP}.
\begin{figure}[htbp]
  \includegraphics[width=9cm,height=6cm]{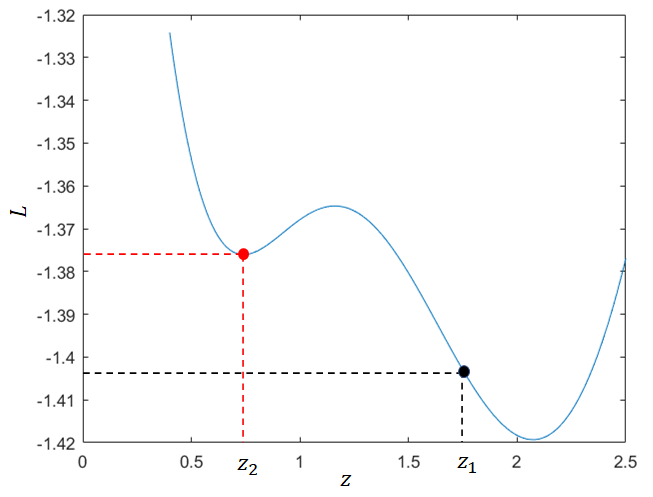}
  \caption{The graph of the function $L(x,\cdot,\bar{u})-0.3x^2$ with $\bar{u}=1.3$}\label{fig}
\end{figure}

(i) Recall that $L(x,\cdot,\bar{u})$ is $\mathcal F$-pseudoconvex over $[x,+\infty)$ if there exists a sublinear function ${\mathcal F}(z_1, z_2; \cdot)$, depend on $z_1, z_2\in [x,+\infty)$, such that
\begin{eqnarray*}
{\mathcal F}\big(z_1, z_2; \nabla_z L(x, z_2,\bar{u})\big) \geq 0 \ \Rightarrow \ L(x,z_1,\bar{u}) \geq L(x, z_2,\bar{u}).
\end{eqnarray*}
Take $\bar{u}=1.3$, $z_1=1.75+x$, and select $z_2$ as drawn in Figure \ref{fig} such that $\nabla_z L(x, z_2,\bar{u})=0$. For any sublinear function ${\mathcal F}(z_1, z_2; \cdot)$, by the positive homogeneity, we  have ${\mathcal F}\big(z_1, z_2; \nabla_z L(x, z_2,\bar{u})\big)=0$. However, by direct calculations, we have
$$
L(x,z_1,\bar{u}) \approx -1.403+0.3x^2 <-1.376 +0.3x^2 \approx L(x, z_2,\bar{u}).
$$
This means that $L(x,\cdot,\bar{u})$ is not $\mathcal F$-pseudoconvex on $[x,+\infty)$ for $\bar{u}= 1.3$.

(ii) By substituting with $y=x+s$, $z=x+t$, and eliminating the variable $u$, problem \eqref{ex-BP=WDP-2} can be rewritten as
\begin{eqnarray}\label{ex-BP=WDP-3}
\min  && x-s \nonumber\\
\mbox{s.t.}&& x\geq0, ~s\geq 0,~t\geq 0, \\
&&-e^{-s^2}+0.3s^2\leq -e^{-t^2}(2t^2+1)-0.3t^2.\nonumber
\end{eqnarray}
Consider two functions
\begin{eqnarray*}
\pi(a):=-e^{-a}+0.3a, \qquad \theta(b):=-e^{-b}(2b+1)-0.3b.
\end{eqnarray*}
The function $\pi$ is strictly increasing and hence it has a unique minimal point $\bar{a}=0$ on $[0,\infty)$ with $\pi(\bar{a})=-1$. On the other hand, since $\theta(b)\rightarrow -\infty$ as $b\rightarrow +\infty$ and $\theta(b)\leq 0$ for any $b\geq 0$, the function $\theta$ must attain its maximum on $[0,\infty)$ at the endpoint $\bar{b}=0$ or the stationary points satisfying $\theta'(\hat{b})=0$ with $\hat{b}>0$.
Solving $\theta'(\hat{b})=e^{-\hat{b}}(2\hat{b}-1)-0.3=0$ with $\hat{b}>0$, we have
\begin{eqnarray}\label{ex-BP=WDP-4}
e^{-\hat{b}}=\frac{3}{10(2\hat{b}-1)}
\end{eqnarray}
and $\hat{b}>0.65$ by $e^{-\hat{b}}< 1$.
Substituting \eqref{ex-BP=WDP-4} into $\theta$, we have
$\theta(\hat{b})=-\frac{6\hat{b}^2+3\hat{b}+3}{10(2\hat{b}-1)}.$
Since $\theta(\bar{b})=-1$ and $\hat{b}> 0.65$, we have
\begin{eqnarray*}
\theta(\hat{b})-\theta(\bar{b})=1-\frac{6\hat{b}^2+3\hat{b}+3}{10(2\hat{b}-1)}
=-\frac{6((\hat{b}-\frac{17}{12})^2+\frac{23}{144})}{10(2\hat{b}-1)}<0.
\end{eqnarray*}
This means that $\theta$ attains its maximum on $[0,\infty)$ at $\bar{b}=0$ with $\theta(\bar{b})=-1$.

The above analysis indicates that the last constraint holds only when $s=t=0$. As a result, problem \eqref{ex-BP=WDP-3} has a unique optimal solution $(\tilde{x},\tilde{s},\tilde{t})=(0,0,0)$. Then, problem \eqref{ex-BP=WDP-2} has a unique optimal solution $(\tilde{x},\tilde{y},\tilde{z},\tilde{u})=(0,0,0,0)$ with the optimal value to be $0$. Consequently, the bilevel program \eqref{ex-BP=WDP-1} is still equivalent to its (WDP) reformulation \eqref{ex-BP=WDP-2} in the sense of Theorem \ref{WDP}, although the assumptions of Theorem \ref{WDP} are not satisfied. 
\end{example}

\begin{theorem}\label{bi-optimal}
Let $(\bar{x}, \bar{y}, \bar{z}, \bar{u}, \bar{v})$ be feasible to {\rm(WDP)} and the weak Wolfe duality hold for $(\mathrm{P}_{\bar{x}})$ and {\rm($\mathrm{WD}_{\bar{x}}$)}.
Then, $\bar{y}$ is globally optimal to {\rm(P$_{\bar{x}}$)}
and $(\bar{z}, \bar{u}, \bar{v})$ is globally optimal to {\rm($\mathrm{WD}_{\bar{x}}$)}.
\end{theorem}

\begin{proof}
By the feasibility of $(\bar{x}, \bar{y}, \bar{z}, \bar{u}, \bar{v})$ to {\rm(WDP)}, we have
\begin{eqnarray}\label{bi-optimal-1}
\min\limits_{y\in Y(\bar{x})}f(\bar{x},y)\leq f(\bar{x}, \bar{y})\leq L(\bar{x}, \bar{z}, \bar{u}, \bar{v})\leq \max\limits_{(z,u,v)\in W(\bar{x})} L(\bar{x},z,u,v).
\end{eqnarray}
By the weak Wolfe duality assumption, we have $$\min\limits_{y\in Y(\bar{x})}f(\bar{x},y)\ge \max\limits_{(z,u,v)\in W(\bar{x})} L(\bar{x},z,u,v).$$ Therefore, all inequalities in \eqref{bi-optimal-1} should be equalities, which means that $\bar{y}$ is a globally optimal solution to {\rm(P$_{\bar{x}}$)} and $(\bar{z}, \bar{u}, \bar{v})$ is a globally optimal solution to {\rm($\mathrm{WD}_{\bar{x}}$)}.
This completes the proof.
\end{proof}

Theorem \ref{WDP} shows the equivalence between (BP) and (WDP) in the globally optimal sense. Since they are both nonconvex optimization problems, it is necessary to investigate their equivalence in the locally optimal sense. Denote by $U(w)$ some neighborhood of a point $w$.

\begin{theorem}\label{local solution1}
{\rm(i)} Suppose that $(\bar{x},\bar{y})$ is locally optimal to {\rm(BP)} restricted on $U(\bar{x})\times U(\bar{y})$ and {\rm($\mathrm{P}_{\bar{x}}$)} satisfies the Guignard constraint qualification at some optimal solution $\bar{z}$ with $\{\bar{u},\bar{v}\}$ being the corresponding Lagrange multipliers. If $(\mathrm{P}_{x})$ and {\rm($\mathrm{WD}_{x}$)} satisfy the weak Wolfe duality for any $x\in U(\bar{x})$, $(\bar{x},\bar{y},\bar{z},\bar{u},\bar{v}) $ is locally optimal to {\rm(WDP)}.

{\rm(ii)} Conversely, let $(\bar{x},\bar{y},\bar{z},\bar{u},\bar{v}) $ be locally optimal to {\rm($\mathrm{WDP}$)} restricted on $U(\bar{x})\times U(\bar{y}) \times U(\bar{z})\times\mathbb{R}^{p}\times\mathbb{R}^{q}$ with $U(\bar{y}) \subseteq U(\bar{z})$. 
If {\rm($\mathrm{P}_{x}$)} satisfies the Guignard constraint qualification at all points in $S(x)$ for any $x\in U(\bar{x})$, $(\bar{x},\bar{y}) $ is locally optimal to {\rm(BP)}.
\end{theorem}

\begin{proof}
{\rm(i)} Since both $\bar{y}$ and $\bar{z}$ are globally optimal solutions to {\rm($\mathrm{P}_{\bar{x}}$)} and $\{\bar{u},\bar{v}\}$ are the Lagrange multipliers corresponding to $\bar{z}$, we have
\begin{eqnarray*}
f(\bar{x},\bar{y}) = f(\bar{x},\bar{z}) = L(\bar{x},\bar{z},\bar{u},\bar{v})
\end{eqnarray*}
and hence $(\bar{x},\bar{y},\bar{z},\bar{u},\bar{v}) $ is feasible to {\rm(WDP)}.

On the other hand, for any feasible point $(x,y,z,u,v)$ to {\rm(WDP)} in the neighborhood $U(\bar{x})\times U(\bar{y})\times U(\bar{z})\times U(\bar{u})\times U(\bar{v})$ with $U(\bar{z})\times U(\bar{u})\times U(\bar{v})$ to be a neighbourhood of $(\bar{z},\bar{u},\bar{v})$,
by the weak Wolfe duality assumption,  
the conclusion of Theorem \ref{bi-optimal} is true and hence $y\in S(x)$, which means that $(x,y)$ is feasible to {\rm(BP)} and belongs to $U(\bar{x})\times U(\bar{y})$. Noting that $(\bar{x},\bar{y})$ is locally optimal to {\rm(BP)} restricted on $U(\bar{x})\times U(\bar{y})$, we have
\begin{eqnarray}\label{local solution1-1}
F(\bar{x},\bar{y}) \leq F(x, y).
\end{eqnarray}
By the arbitrariness of $(x,y,z,u,v)$, we get that $(\bar{x},\bar{y},\bar{z},\bar{u},\bar{v}) $ is a locally optimal solution to {\rm(WDP)}.

{\rm(ii)} Let $(x,y)\in U(\bar{x})\times U(\bar{y})$ be feasible point to {\rm(BP)}. It follows that $y\in S(x)$. By the assumptions, {\rm($\mathrm{P}_{x}$)} satisfies the Guignard constraint qualification at $y$ and hence there exist multipliers $\{u,v\}$ satisfying
\begin{eqnarray*}
\nabla_y L(x, y,u,v) = 0, ~u\geq 0, ~u^{T} g(x, y)=0.
\end{eqnarray*}
It then follows from $h(x,y)=0$ that $f(x, y)=L(x, y,u,v)$. In consequence, $(x,y,y,u,v)$ is feasible to {\rm(WDP)} and belongs to $U(\bar{x})\times U(\bar{y})\times U(\bar{y})\times U(\bar{u})\times U(\bar{v})$. Since $(\bar{x},\bar{y},\bar{z},\bar{u},\bar{v}) $ is locally optimal to {\rm($\mathrm{WDP}$)} restricted on $U(\bar{x})\times U(\bar{y}) \times U(\bar{z})\times\mathbb{R}^{p}\times\mathbb{R}^{q}$ with $U(\bar{y}) \subseteq U(\bar{z})$, we have \eqref{local solution1-1}. By the arbitrariness of $(x,y)$,
$(\bar{x},\bar{y}) $ is a locally optimal solution to {\rm(BP)}. This completes the proof.
\end{proof}

The local optimality restricted on $U(\bar{x})\times U(\bar{y}) \times U(\bar{z})\times \mathbb{R}^{p}\times\mathbb{R}^{q}$ assumed in Theorem \ref{local solution1}(ii) is obviously very stringent. In the rest of this subsection, we derive a result by restricting the local optimality on $U(\bar{x})\times U(\bar{y}) \times U(\bar{y})\times U(\bar{u})\times U(\bar{v})$. We start with the following lemma.

\begin{lemma}\label{lemma-usc}
The set-valued mapping $\Lambda : \mathbb{R}^{n+2m}\rightrightarrows \mathbb{R}^{p+q}$ defined by
\begin{eqnarray}\label{Lambda}
\Lambda(x, y, z) := \big\{(u, v)\in\mathbb{R}^{p+q} : (x, y, z, u, v)\in {\Omega} \big\}
\end{eqnarray}
is outer semicontinuous, 
where
\begin{eqnarray}\label{stable system}
{\Omega}:=\left\{ (x, y, z, u, v) :
\begin{array}{l}
\nabla _z L(x, z, u, v) = 0, ~u\ge 0\\
f(x, y) - L(x, z, u, v) \leq 0
\end{array}
\right\}.
\end{eqnarray}
\end{lemma}

\begin{proof} By Theorem 5.7 in \cite{Rockafellar2009variational}, $\Lambda$ is outer semicontinuous if and only if $\Lambda^{-1}(B)$ is closed for every compact set $B\subset \mathbb{R}^{p+q}$, which can be obtained by the twice continuous differentiability of the functions $\{f, g, h\}$.
\end{proof}

\begin{theorem}
Let $(\bar{x},\bar{y},\bar{y},\bar{u},\bar{v}) $ be locally optimal to {\rm(WDP)} for any $(\bar{u},\bar{v})\in \Lambda(\bar{x},\bar{y},\bar{y})$ and the weak Wolfe duality hold for $(\mathrm{P}_{\bar{x}})$ and {\rm($\mathrm{WD}_{\bar{x}}$)}. If {\rm($\mathrm{P}_{{x}}$)} satisfies the MFCQ at ${y}\in S({x})$ for all $(x,y)$ closely to $(\bar{x},\bar{y})$, $(\bar{x},\bar{y})$ is locally optimal to {\rm(BP)}.
\end{theorem}

\begin{proof}
Since $(\bar{x},\bar{y},\bar{y},\bar{u},\bar{v}) $ is feasible to {\rm(WDP)}, 
it follows from the weak Wolfe duality assumption and Theorem \ref{bi-optimal} that 
$\bar{y}\in S(\bar{x})$ and hence $(\bar{x},\bar{y})$ is  feasible to {\rm(BP)}.
We next show the result by contradiction.

Suppose that $(\bar{x},\bar{y})$ is not locally optimal to {\rm(BP)}. There must exist a sequence $\{(x^k, y^k)\}$ converging to $(\bar{x},\bar{y})$ with $x^k\in X$, $y^k\in S(x^k)$ such that
\begin{eqnarray}\label{local-inequality}
F(x^k, y^k) < F(\bar{x},\bar{y}), \quad k=1,2,\cdots.
\end{eqnarray}
By the assumptions, when $k$ is large sufficiently, {\rm($\mathrm{P}_{{x^k}}$)} satisfies the MFCQ at ${y^k}\in S({x^k})$, there exist multipliers $u^k\in\mathbb{R}^p$ and $v^k\in\mathbb{R}^q$ such that
\begin{eqnarray}
\label{stable system1}&&\nabla _y L(x^k, y^k, u^k, v^k) = 0, ~~h(x^k, y^k)=0,\\
\label{stable system2}&&u^k\ge 0,~g(x^k, y^k) \leq 0,~(u^k)^Tg(x^k, y^k)=0 .
\end{eqnarray}
We next show that the sequence of $\{(u^k,v^k)\}$ is bounded.

Suppose by contradiction that $\{(u^k,v^k)\}$ is unbounded. Without loss of generality, we may assume that
$\{\frac{u^k}{\|u^k\|+\|v^k\|}\}$ and $\{\frac{v^k}{\|u^k\|+\|v^k\|}\}$ are  convergent.
Denote their limits by $\hat{u}$ and $\hat{v}$ respectively.
It follows that
\begin{eqnarray}\label{uv-inequality}
\|\hat{u}\|+\|\hat{v}\|=1.
\end{eqnarray}
Dividing by $\|{u}^k\|+\|{v}^k\|$ on both sides of \eqref{stable system1}-\eqref{stable system2} and letting $k\rightarrow\infty$ yield
\begin{eqnarray}
 \label{stable system3-1}&&\nabla_y g(\bar{x},\bar{y})\hat{u}+\nabla_y h(\bar{x},\bar{y})\hat{v}=0,\\
\label{stable system3-2}&&\hat{u}\geq 0, ~g(\bar{x},\bar{y})\le 0, ~\hat{u}^{T} g(\bar{x},\bar{y}) = 0.
\end{eqnarray}
Let ${\cal I}_y$ be the active index set of $g$ at $(\bar{x},\bar{y})$. By \eqref{stable system3-2}, \eqref{stable system3-1} reduces to
\begin{eqnarray}
 \label{stable system3-3}
 \sum_{i\in {\cal I}_y} \hat{u}_i\nabla_y g_i(\bar{x},\bar{y})+\sum_{j=1}^q \hat{v}_j\nabla_y h_j(\bar{x},\bar{y})=0.
\end{eqnarray}
Since the MFCQ holds at $\bar{y}$ for {\rm($\mathrm{P}_{\bar{x}}$)}, the gradients $\nabla_y h_j(\bar{x},\bar{y})~(j=1, \cdots, q) $
are linearly independent and there exists $ d\in\mathbb{R}^m $ such that
$d^T\nabla _y h_j(\bar{x},\bar{y}) =0$
for each $j$ and 
$d^T \nabla_y g_i(\bar{x}, \bar{y})<0$ for each $i\in{\cal I}_y$. 
Taking the inner product on both sides of \eqref{stable system3-3} with $d$, we have
\begin{eqnarray*}
0=\sum_{i\in {\cal I}_y} \hat{u}_i d^T\nabla_y g_i(\bar{x},\bar{y})<0,
\end{eqnarray*}
which implies ${\cal I}_y=\emptyset$ and hence $\hat{u}=0$ by \eqref{stable system3-2}. Thus, \eqref{stable system3-1} becomes $\nabla_y h(\bar{x},\bar{y})\hat{v}=0$, which implies $\hat{v}=0$ because the gradients $ \nabla_y  h_j(\bar{x},\bar{y})~(j=1, \cdots, q) $ are linearly independent.
This contradicts \eqref{uv-inequality}. Therefore, the sequence of $\{(u^k,v^k)\}$ is bounded. 

Since $(x^k, y^k, y^k) \to (\bar{x},\bar{y},\bar{y})$ as $k \to \infty$,
by the outer semicontinuousness of $\Lambda$ from Lemma \ref{lemma-usc}, the sequence $\{(u^k, v^k)\}$ has at least an accumulation point $(\tilde{u},\tilde{v})\in \Lambda(\bar{x},\bar{y},\bar{y})$ as $k\to\infty$.
By the assumptions, the limit point $(\bar{x}, \bar{y}, \bar{y}, \tilde{u}, \tilde{v})$ is also locally optimal to {\rm(WDP)}. Note that, from the definition of $\Lambda$, $(x^k, y^k, y^k, u^k, v^k)$ is feasible to {\rm(WDP)} for each $k$.
Thus, the inequality (\ref{local-inequality}) is a contradiction. Consequently, $(\bar{x},\bar{y})$ must be a locally optimal solution to {\rm(BP)}. This completes the proof.
\end{proof}

%%%%%%%%%%%%%%%%%%%%%%%%%%%%%%%%
%%%%%%%%%%%%%%%%%%%%%%%%%%%%%%%%
%%%%%%%%%%%%%%%%%%%%%%%%%%%%%%%%
\section{Comparison between WDP and MPEC}

In the previous section, it is shown that, under mild assumptions, the bilevel program (BP) is equivalent to the single-level optimization problem (WDP) in the globally or locally optimal sense. Note that, unlike the reformulations (SP), (VP), and (LDP) introduced in Section 2, (WDP) does not involve any function without analytic expression. In this section, our first purpose is to compare (WDP) with the most popular reformulation MPEC to reveal which is better and our second purpose is to investigate their relations.

\subsection{WDP may satisfy the MFCQ at its feasible points}

As stated before, in order to solve the bilevel program (BP) effectively, it is generally required to transform it into some single-level optimization problem, no matter whether it satisfies those conditions to ensure their equivalence in theory. So, which is better between (WDP) and MPEC?

Recall that MPEC fails to satisfy the MFCQ at any feasible point. This means that the well-developed optimization algorithms in nonlinear programming may not be stable in solving MPEC. What about (WDP)?

\begin{example}\rm\label{ex1}
Consider the bilevel program
\begin{eqnarray}
\min &&(x-y-8)^2  \label{ex1-1}\\
\mbox{s.t.} && x\geq 1,~y\in {S}(x)=\arg\min\limits_{y}\{y: y^3\leq x, y\geq0\}.\nonumber
\end{eqnarray}
Obviously, ${S}(x)=\{0\}$ for any $x\geq 1$ and so \eqref{ex1-1} has a unique minimizer $(8, 0)$.
The (WDP) reformulation for \eqref{ex1-1} is
\begin{eqnarray}
\min && (x-y-8)^2  \nonumber\\
\mbox{s.t.} && x\geq 1, ~y^3 - x \leq 0,~y\geq 0, \label{ex1-5-duality}\\
&& y - z - u_1(z^3-x) + u_2 z \leq 0,\nonumber\\
&&1 + 3u_1 z^2 - u_2=0, ~u_1\geq 0, ~u_2\geq 0.\nonumber
\end{eqnarray}
Consider the feasible point $(x^*, y^*, z^*, u_1^*, u_2^*):=(8, 0, -3, 0, 1)$, which is actually a globally optimal solution of \eqref{ex1-5-duality}.
In what follows, we show that the MFCQ holds at this point.
\begin{itemize}
\item The gradient corresponding to the unique equality constraint at the point is
$\mathbf{g}_0 :=(0, 0, 0, 27, -1)^T$, which is a nonzero vector, and hence it is linearly independent.

\item The active inequality constraints at the point include
$$y \geq 0, \quad y - z - u_1(z^3-x) + u_2 z \leq 0, \quad u_1 \geq 0.$$
The corresponding gradients at the point are respectively given by
\begin{eqnarray*}\label{ex1-7}
&&\mathbf{g}_1 := (0, -1, 0, 0, 0)^T,\\[1mm]
&&\mathbf{g}_2 :=(0, 1, 0, 35, -3)^T,~~~~\\[1mm]
&&\mathbf{g}_3 := (0, 0, 0, -1, 0)^T.
\end{eqnarray*}
Take a vector $d:=(0,1,0,1,27)^{T}$. It is trivial to verify that $d^T \mathbf{g}_0=0$, $d^T \mathbf{g}_1< 0$, $d^T \mathbf{g}_2< 0$, and $d^T \mathbf{g}_3< 0$.
\end{itemize}
The above analysis indicates that \eqref{ex1-5-duality} satisfies the MFCQ at its feasible point $(8, 0, -3, 0, 1)$.
\end{example}

Example \ref{ex1} shows that (WDP) may satisfy the MFCQ at its feasible points. In consequence, compared with MPEC, (WDP) is relatively easy to deal with.

%%%%%%%%%%%%%%%%%%%%%%%%%%%%%%%%
%%%%%%%%%%%%%%%%%%%%%%%%%%%%%%%%
%%%%%%%%%%%%%%%%%%%%%%%%%%%%%%%%
\subsection{Relations between WDP and MPEC}
\label{optimality}

In this subsection, for simplicity, we take the upper constraint $x\in X$ and the lower equality constraint $h(x,y)=0$ away from the original program (BP). In this case, the corresponding MPEC becomes
\begin{eqnarray}\label{simplifiedMPEC}
\min  && F(x,y) \nonumber\\
\mbox{s.t.\!\! } && \nabla_y L(x, y, u)=0, \\
&& 0\leq u ~ \bot ~g(x, y)\le 0,\nonumber
\end{eqnarray}
where $L(x,y,u):=f(x, y) + u^{T} g(x, y)$, while (WDP) becomes
\begin{eqnarray}\label{simplifiedWDP}
\min &&F(x, y) \nonumber \\
\mbox{\rm s.t.\! } && f(x, y) - L(x,z,u) \leq 0, \\
&&\nabla_z L(x, z, u)= 0, ~u\geq 0,~g(x, y) \le 0. \nonumber
\end{eqnarray}
In what follows, we investigate the relations between the stationarity conditions of the above problems. There is no  difficulty to extend the subsequent analysis to the general MPEC and the general (WDP).

By simple deduction, we have the following result immediately.

\begin{theorem}\label{mpcc and wdp feasible}
The point $(x, y, u)$ is feasible to \eqref{simplifiedMPEC} if and only if $(x, y, y, u)$ is feasible to \eqref{simplifiedWDP}.
\end{theorem}

Let $(\bar{x},\bar{y},\bar{u})$ be a feasible point of \eqref{simplifiedMPEC}. We define the following index sets:
\begin{eqnarray*}
&&I_{0+}:=\{i : g_{i}(\bar{x},\bar{y})=0, ~\bar{u}_i>0\},\\
&&I_{-0}:=\{i : g_{i}(\bar{x},\bar{y})<0, ~\bar{u}_i=0\},\\
&&I_{00}:=\{i : g_{i}(\bar{x},\bar{y})=0, ~\bar{u}_i=0\}.
\end{eqnarray*}

\begin{definition}\rm
We call $ (\bar{x},\bar{y},\bar{u})$ a strongly stationary point (or S-stationary point) of \eqref{simplifiedMPEC}
if there exists $(\lambda^g,\lambda^u,\gamma)\in\mathbb{R}^{2p+m}$
such that
\begin{eqnarray}
&& \nabla_x F(\bar{x},\bar{y}) + \nabla_{yx}^2 L(\bar{x},\bar{y}, \bar{u}) \gamma  + \nabla_x g(\bar{x},\bar{y}) \lambda^g=0,\label{MPCC-1}\\
&&\nabla_y F(\bar{x},\bar{y}) + \nabla_{yy}^2 L(\bar{x},\bar{y}, \bar{u})\gamma + \nabla_y g(\bar{x},\bar{y}) \lambda^g=0, \label{MPCC-2}\\
&&
\nabla_y g(\bar{x},\bar{y})^T \gamma -\lambda^u= 0, \label{MPCC-3}\\
&& \nabla_{y} L(\bar{x},\bar{y}, \bar{u}) =0, ~g(\bar{x},\bar{y})\le 0, ~\bar{u}\ge 0,\label{MPCC-4}\\
&&
\lambda^g_i = 0, ~~i \in I_{-0}, \label{MPCC-6}\\
&&
\lambda^u_i = 0, ~~i \in I_{0+}, \label{MPCC-7}\\
&&\lambda^g_i \geq 0, ~\lambda^u_i \geq 0, ~~i\in I_{00}. \label{MPCC-8}
\end{eqnarray}
\end{definition}

\begin{remark}\rm
Other popular stationarity conditions for MPEC include the Mordukhovich stationarity (or M-stationarity) and the Clarke stationarity (or C-stationarity). Among these stationarities, the S-stationarity is the strongest and the most favorable one in the study of MPEC.
\end{remark}

\begin{theorem}\label{kkt and S}
Let $(\bar{x},\bar{y},\bar{y},\bar{u})$ be a KKT point of \eqref{simplifiedWDP}. Then, $(\bar{x}, \bar{y}, \bar{u})$ is an S-stationary point of \eqref{simplifiedMPEC}.
\end{theorem}

\begin{proof}
If $(\bar{x},\bar{y},\bar{y},\bar{u})$ is a KKT point of \eqref{simplifiedWDP}, there exists
$(\eta^g, \eta^u, \alpha, \beta) \in \mathbb{R}^{2p+1+m}$ such that 
\begin{eqnarray}
&&\nabla_x  F(\bar{x},\bar{y})  + \nabla_{yx}^2 L(\bar{x},\bar{y}, \bar{u}) \beta  +\nabla_x g(\bar{x},\bar{y}) (\eta^g-\alpha\bar{u})=0, \label{WDPKKT-1} \\
&&\nabla_y F(\bar{x},\bar{y}) + \alpha\nabla_y f(\bar{x},\bar{y}) + \nabla_y g(\bar{x},\bar{y})\eta^g=0, \label{WDPKKT-2}\\
&&-\alpha\nabla_y L(\bar{x},\bar{y}, \bar{u})+\nabla_{yy}^2 L(\bar{x},\bar{y}, \bar{u}) \beta=0, \label{WDPKKT-3}\\
&&-\alpha g(\bar{x},\bar{y}) + \nabla_y g(\bar{x},\bar{y})^T\beta - \eta^u=0, \label{WDPKKT-4}\\
&&\nabla_y L(\bar{x},\bar{y}, \bar{u})=0,  \label{WDPKKT-6}\\
&&0\le \alpha ~\bot ~\bar{u}^{T} g(\bar{x},\bar{y}) \geq0,  \label{WDPKKT-8}\\
&&0\le \eta^g ~\bot ~g(\bar{x},\bar{y})\leq 0,  \label{WDPKKT-9}\\
&&0\le \eta^u ~\bot ~\bar{u} \geq 0.  \label{WDPKKT-10}
\end{eqnarray}
Adding \eqref{WDPKKT-2} and \eqref{WDPKKT-3} together yields
\begin{eqnarray}
\nabla_y F(\bar{x},\bar{y})+\nabla_{yy}^2 L(\bar{x},\bar{y}, \bar{u}) \beta + \nabla_y g(\bar{x},\bar{y})(\eta^g-\alpha\bar{u})=0. \label{WDPKKT-5}
\end{eqnarray}
Set
$\lambda^g:=\eta^g-\alpha\bar{u}, ~\lambda^u:=\eta^u+\alpha g(\bar{x},\bar{y})$ and $\gamma:=\beta$.
We obtain \eqref{MPCC-1}-\eqref{MPCC-4} from \eqref{WDPKKT-1} and \eqref{WDPKKT-4}-\eqref{WDPKKT-5} immediately. Next, we show \eqref{MPCC-6}-\eqref{MPCC-8}.
\begin{itemize}
\item If $i\in I_{-0}$, which means $g_{i}(\bar{x},\bar{y})<0$ and $\bar{u}_i=0$, we have $\eta_i^g=0$ by \eqref{WDPKKT-9} and hence $\lambda_i^g=\eta_i^g-\alpha\bar{u}_i=0$. Namely, \eqref{MPCC-6} holds.

\item If $i\in I_{0+}$, which means $g_{i}(\bar{x},\bar{y})=0$ and $\bar{u}_i>0$, we have $\eta_i^u=0$ by \eqref{WDPKKT-10} and hence $\lambda_i^u=\eta_i^u+\alpha g_{i}(\bar{x},\bar{y})=0$. Namely, \eqref{MPCC-7} holds.

\item If $i\in I_{00}$, which means $g_{i}(\bar{x},\bar{y})=\bar{u}_i=0$, we have $\lambda_i^g=\eta_i^g\geq 0$ and $\lambda_i^u=\eta_i^u \geq 0$ by \eqref{WDPKKT-9} and \eqref{WDPKKT-10} respectively. Namely, \eqref{MPCC-8} holds.
\end{itemize}
This completes the proof.
\end{proof}

The following example shows that the converse of Theorem \ref{kkt and S} is not true.

\begin{example}\rm
Consider the bilevel program
\begin{eqnarray}\label{ex2-1}
\min &&x^2-(2y+1)^2\nonumber\\
\mbox{s.t.}&& x\leq0,\\
&&y\in \arg\min\limits_{y}\{(y-1)^2: 3x-y-3\leq0,~x+y-1\leq0\}.\nonumber
\end{eqnarray}
It is easy to see that $(0,1)$ is the unique globally optimal solution of \eqref{ex2-1}.
Since the lower-level program is convex and satisfies the linear constraint qualification everywhere, we can equivalently transform \eqref{ex2-1} into the MPEC
\begin{eqnarray}
\min &&x^2-(2y+1)^2 \nonumber\\
\mbox{s.t.} &&x\leq0, ~2(y-1)-u_1+u_2=0,\label{ex2-3}\\
&&0\le u_1 ~\bot~3x-y-3\leq0, ~0\le u_2 ~\bot ~x+y-1\leq0,\nonumber
\end{eqnarray}
or (WDP)
\begin{eqnarray}
\min &&x^2-(2y+1)^2 \nonumber\\
\mbox{s.t.\! } &&x\leq0, ~2(z-1)-u_1+u_2=0,\label{ex2-4}\\
&&(y-1)^2-(z-1)^2-u_1(3x-z-3)-u_2(x+z-1)\leq 0, \nonumber\\
&&u_1\ge 0,~u_2\geq0, ~3x-y-3\leq0, ~x+y-1\leq0. \nonumber
\end{eqnarray}

Consider the feasible point $(0,1,0,0)$ of the MPEC \eqref{ex2-3}. The S-stationarity conditions at this point are
\begin{eqnarray*}
&& \alpha+3 \lambda_1^g+ \lambda_2^g=0,~~-12 + 2\beta - \lambda_1^g+ \lambda_2^g=0, \nonumber\\
&&
- \beta -\lambda_1^u= 0,  ~~\beta -\lambda_2^u= 0, ~~\alpha\geq0,\label{ex2-5}\\
&&  \lambda^g_1 = 0, ~~\lambda^g_2 \geq 0, ~~\lambda^u_2 \geq 0, \nonumber
\end{eqnarray*}
where $\{\alpha, \beta\}$ are multipliers corresponding to the first two constraints in \eqref{ex2-3} and $\{\lambda^u,\lambda^g\}$ are multipliers corresponding to the complementarity constraints. It is obvious that $\alpha=0, \beta=6, \lambda^u=(-6,6),$ and $\lambda^g=(0,0)$ satisfy the above conditions and hence $(0,1,0,0)$ is an S-stationary point of the MPEC \eqref{ex2-3}.

Next, we consider the feasible point $(0,1,1,0,0)$ of problem \eqref{ex2-4}.
The KKT conditions at this point are
\begin{eqnarray*}
&&\xi + 3\mu_1^g+\mu_2^g=0, ~~-12  -\mu_1^g+\mu_2^g=0, ~~2\eta=0, \label{ex2-6-3}\\
&&-\eta+4\zeta-\mu_1^u=0, ~~\eta-\mu_2^u=0,  \label{ex2-6-5}\\
&&\xi\geq0, ~~\zeta \geq 0, ~~\mu^g_1 = 0, ~~\mu^g_2 \geq 0, ~~\mu^u \geq 0,  \label{ex2-6-9}
\end{eqnarray*}
where $\{\xi,\eta, \zeta\}$ are multipliers corresponding to the first three constraints in \eqref{ex2-4} and $\{\mu^u,\mu^g\}$ are multipliers corresponding to the others. Note that
$$\mu^g_1 = 0 \quad \Rightarrow \quad \mu^g_2 = 12 \quad \Rightarrow \quad \xi=-12,$$
which contradicts $\xi\geq0$. Therefore, $(0,1,1,0,0)$ is not a KKT point of \eqref{ex2-4}. 
\end{example}

We further have the following result.

\begin{theorem}\label{MFCQ-Fail}
Let $(x, y, z, u)$ be a feasible point of \eqref{simplifiedWDP}. Then, the MFCQ fails to hold at this point if $z=y$.
\end{theorem}

\begin{proof}
The set of abnormal multipliers for \eqref{simplifiedWDP} at $(x, y, z, u)$ is given by
\begin{eqnarray*}
\mathcal{M}:=\left\{\begin{array}{lll}
&{0= \alpha(\nabla_x f(x, y)-\nabla_x L(x, z,u))}\\
{}&{~~~~~ +\nabla^2_{zx}L(x,z,u)\beta+\nabla_x g(x, y) \eta^g}\\[1mm]
{}&{0=\alpha\nabla_y f(x, y)+ \nabla_y g(x, y)\eta^g}\\[1mm]
{(\alpha, \beta,  \eta^g,  \eta^u):\ }&{0=-\alpha\nabla_z L(x, z,u)+\nabla^2_{zz}L(x,z,u)\beta}\\
&{0=-\alpha g(x, z) +\nabla_z g(x, z)^{T}\beta-\eta^u}\\[1mm]
{}&{\alpha\geq 0,~\alpha(f(x, y)-L(x, z,u))=0} \\[1mm]
{}&{\eta^u\geq 0,~u^T \eta^u =0, ~\eta^g\geq 0,~ g(x, y)^T\eta^g =0}
\end{array}\right\}.
\end{eqnarray*}
If $z=y$, it is not difficult to verify by the feasibility of $(x, y, y, u)$ to \eqref{simplifiedWDP} that $\big( 1, 0, u, -g(x, y) \big)\in \mathcal{M}$. Recall that, by \cite[Chapter 6.3]{Clarke1990optimization}, a nonzero abnormal multiplier exists if and only if the MFCQ does not hold. Therefore, the MFCQ does not hold at $(x, y, y, u)$. This completes the proof.
\end{proof}

%%%%%%%%%%%%%%%%%%%%%%%%%%%%%%%%
%%%%%%%%%%%%%%%%%%%%%%%%%%%%%%%%
%%%%%%%%%%%%%%%%%%%%%%%%%%%%%%%%
\section{Relaxation scheme for WDP}
\label{relaxation}

Theorem \ref{MFCQ-Fail} indicates that, although (WDP) may satisfy the MFCQ at its feasible points, it is still difficult to solve directly in some cases. Moreover, our numerical experiments given in Subsection \ref{directsolver} show that the efficiency was not very high if we solved the bilevel program (BP) by solving (WDP) directly. Based on these reasons, we present a relaxation scheme to deal with (WDP) in this section.

Our strategy is to use the following problem to approximate (WDP):
\begin{eqnarray*}
\mathrm{WDP}(t)\quad\min &&F(x,y) \\
\mbox{s.t.}&&x\in X,~g(x,y)\le0,~h(x,y)=0,\\
&&f(x,y)-L(x, z, u, v) \leq t,\\
&&\nabla_z L(x, z, u, v) = 0,~u\geq0,
\end{eqnarray*}
where $t$ is a positive parameter. Obviously, all feasible points of (WDP) are also feasible to the above problem and, in addition, the second inequality constraint is not active any more at these points so that the popular constraint qualifications may be easily satisfied. Therefore, we may solve $\mathrm{WDP}(t)$ as a constrained optimization problem directly.

We first give some results related to constraint qualifications. To this end, we introduce the following definition.

\begin{definition}\rm\cite{Davis1954theory}\label{linnear-independent} The vectors $\{a^1,\cdots,a^l\}\cup\{b^1,\cdots,b^r\}$ are called to be positive-linearly dependent if there exist $\alpha\in \mathbb{R}^l_+$ and $\beta\in \mathbb{R}^r$ with $(\alpha, \beta)\not=0$ such that
$$ \sum_{i=1}^l \alpha_i a^i +  \sum_{i=1}^r \beta_i b^i=0.$$
Otherwise, the set is called to be positive-linearly independent.
\end{definition}

By \cite[Proposition 1.8.2]{Polak1997optimization}, the MFCQ is satisfied at a feasible point if and only if the gradients of all active constraints at the point 
are positive-linearly independent. By \cite[Proposition 2.2]{Qi2000constant}, for given continuous functions $\{G_1(w), \cdots, G_l(w)\}\cup\{H_1(w), \cdots, H_r(w)\}$, if they are positive-linearly independent at some point $\bar{w}$, there exists a neighborhood $U(\bar{w})$ such that they are positive-linearly independent at any point in $U(\bar{w})$.

From the above introduction, we have the following result immediately.

\begin{theorem}
If $\mathrm{WDP}(t)$ satisfies the MFCQ at a feasible point $(\bar{x},\bar{y},\bar{z},\bar{u},\bar{v})$,
there exists a neighborhood $\bar{U}$ of $(\bar{x},\bar{y},\bar{z},\bar{u},\bar{v})$ such that $\mathrm{WDP}(t)$ satisfies the MFCQ at any feasible point in $\bar{U}$.
\end{theorem}

The above theorem can be extended to the following constraint qualification, which is weaker than the MFCQ and has attracted lots of attention recently.

\begin{definition}\rm \cite{Qi2000constant}
Let $\bar{y}\in {Y}$ be a feasible point of the constrained optimization problem (P) in Subsection \ref{Wolfe-duality}. We say that {\rm (P)} satisfies the
constant positive linear dependence (CPLD) at $\bar{y}$ if, for any $I_0\subseteq \{ i : g_i(\bar{y})=0\}$ and $J_0\subseteq \{1,\cdots, q\}$ such that $\{\nabla g_i(\bar{y}), i\in I_0\}\cup\{\nabla h_j(\bar{y}), j\in J_0\}$ is positive-linearly dependent, there exists a neighborhood $U(\bar{y})$ such that $\{\nabla g_i(y), i\in I_0; ~\nabla h_j(y), j\in J_0\}$ is linearly dependent for any $y\in U(\bar{y})$.
\end{definition}

\begin{theorem}\label{CPLDth}
If $\mathrm{WDP}(t)$ satisfies the CPLD at a feasible point $(\bar{x},\bar{y},\bar{z},\bar{u},\bar{v})$,
there exists a neighborhood $\bar{U}$ of $(\bar{x},\bar{y},\bar{z},\bar{u},\bar{v})$ such that $\mathrm{WDP}(t)$ satisfies the CPLD at any feasible point in $\bar{U}$.
\end{theorem}

\begin{proof} For simplicity, we take the upper constraint $x\in X$ away from $\mathrm{WDP}(t)$ and denote by $\bar{w}:=(\bar{x},\bar{y},\bar{z},\bar{u},\bar{v})$, $w:=(x,y,z,u,v)$. We further denote by $G_i(w) ~(i=0,1,\cdots,2p)$ all inequality constraint functions and by $H_i(w) ~(i=1,\cdots,m+q)$ all equality constraint functions.

Suppose by contradiction that the conclusion is wrong. That is, there is a sequence $\{w^k\}$ of feasible points tending to $\bar{w}$ such that $\mathrm{WDP}(t)$ does not satisfy the CPLD at each $w^k$. 
Violation of CPLD means that there exist $I_0^k\subseteq \{i : G_i(w^k)=0\}$ and $J_0^k\subseteq \{1,\cdots,m+q\}$ such that $\{\nabla G_i(w^k), i\in I_0^k\}\cup\{\nabla H_j(w^k), j\in J_0^k\}$ is positive-linearly dependent, but linearly independent in some points arbitrarily close to $w^k$. Note that the number of subsets of the index sets $\{i : G_i(w^k)=0\}$ and $\{1,\cdots,m+q\}$ is finite. Taking a subsequence if necessary, we may assume $I_0^k\equiv I_0$ and $J_0^k\equiv J_0$. Therefore, there is another sequence $\{\tilde{w}^k\}$ tending to $\bar{w}$, together with $\{w^k\}$, satisfying
\begin{itemize}
\item[(i)] $\{\nabla G_i(w^k), i\in I_0\}\cup\{\nabla H_j(w^k), j\in J_0\}$ is positive-linearly dependent;

\item[(ii)] $\{\nabla G_i(\tilde{w}^k), i\in I_0; ~\nabla H_j(\tilde{w}^k), j\in J_0\}$ is linearly independent.
\end{itemize}

If $\{\nabla G_i(\bar{w}), i\in I_0\}\cup\{\nabla H_j(\bar{w}), j\in J_0\}$ is positive-linearly dependent, by the CPLD assumption, there exists a neighborhood which remains the linear dependence, which contradicts the existence of $\{\tilde{w}^k\}$ in (ii). If $\{\nabla G_i(\bar{w}), i\in I_0\}\cup\{\nabla H_j(\bar{w}), j\in J_0\}$ is positive-linearly independent, by \cite[Proposition 2.2]{Qi2000constant}, this property should remain in a neighborhood, which contradicts the existence of $\{w^k\}$ in (i). As a result, the conclusion is true. This completes the proof.
\end{proof}

\begin{remark}\rm
It is easy to see from the proofs that the above results related to constraint qualifications remain true if $\mathrm{WDP}(t)$ is replaced by (WDP).
\end{remark}

In what follows, we consider the limiting behavior of ${\mathrm{WDP}(t)}$ as the approximation parameter $t$ tends to $0$. The following convergence result related to globally optimal solutions is evident.

\begin{theorem}
Let $\{t_k\}\downarrow0$ and $(x^k,y^k,z^k,u^k,v^k)$ be a globally optimal solution of ${\mathrm{WDP}(t_k)}$ for each $k$. Then, every accumulation point $(\bar{x},\bar{y},\bar{z},\bar{u},\bar{v})$ of the sequence $\{(x^k,y^k,z^k,u^k,v^k)\}$ is globally optimal to ${\rm (WDP)}$.
\end{theorem}

Since ${\mathrm{WDP}(t)}$ is usually a nonconvex optimization problem, it is necessary to investigate the limiting properties of its stationary points.

\begin{theorem}
Let $\{t_k\}\downarrow0$ and $(x^k,y^k,z^k,u^k,v^k)$ be a KKT point of ${\mathrm{WDP}(t_k)}$ for each $k$.
Assume that $(\bar{x},\bar{y},\bar{z},\bar{u},\bar{v})$ is an accumulation point of the sequence $\{(x^k,y^k,z^k,u^k,v^k)\}$. If {\rm (WDP)} satisfies the CPLD at $(\bar{x},\bar{y},\bar{z},\bar{u},\bar{v})$, then $(\bar{x},\bar{y},\bar{z},\bar{u},\bar{v})$ is a KKT point of ${\rm (WDP)}$.
\end{theorem}

\begin{proof} For simplicity, we take the upper constraint $x\in X$ away from (WDP) and $\mathrm{WDP}(t)$. In addition, we denote
\begin{eqnarray*}
&&\bar{w}:=(\bar{x},\bar{y},\bar{z},\bar{u},\bar{v}), ~~~w^k:=(x^k,y^k,z^k,u^k,v^k), ~~~w:=(x,y,z,u,v),\\
&&F_0(w):=F(x,y), ~~~G_0(w):=f(x,y)-L(x, z, u, v), \\
&&G_i(w):=g_i(x,y) ~(1\le i\le p),  ~~~G_i(w):=-u_{i-p} ~(p+1\le i\le 2p),\\
&&H_i(w):=h_i(x,y) ~(1\le i\le q),  ~~H_i(w):=\nabla_{z_{i-q}}L(x,z,u,v) ~(q+1\le i\le q+m).
\end{eqnarray*}
Then, (WDP) and $\mathrm{WDP}(t_k)$ become
\begin{eqnarray}\label{app-WDP}
\min &&F_0(w) \nonumber \\
\mbox{\rm s.t.} && G_i(w) \leq 0, ~i=0,1,\cdots,2p,\\
&&H_i(w) =0, ~i=1,\cdots,q+m, \nonumber
\end{eqnarray}
and
\begin{eqnarray}\label{app-WDPt}
\min &&F_0(w) \nonumber \\
\mbox{\rm s.t.} && G_0(w) \leq t_k, \\
&& G_i(w) \leq 0, ~i=1,\cdots,2p,\nonumber\\
&&H_i(w) =0, ~i=1,\cdots,q+m. \nonumber
\end{eqnarray}
Obviously, $\bar{w}$ is a feasible point of (WDP) and, without loss of generality, we may assume that the whole sequence $\{w^k\}$ converges to $\bar{w}$.

Since $w^k$ is a KKT point of \eqref{app-WDPt}, it follows from \cite[Lemma A.1]{Steffensen2010new} that there exists $(\lambda_0^k,\lambda^k,\mu^k) \in \mathbb{R}\times\mathbb{R}^{2p}\times\mathbb{R}^{q+m}$ such that
\begin{eqnarray}
&&\nabla F_0(w^k)+\lambda_0^k\nabla G_0(w^k)+\nabla G(w^k)\lambda^k+\nabla H(w^k)\mu^k=0,
\nonumber\\
&&\lambda_0^k(G_0(w^k)-t_k)=0, ~~\lambda_0^k\geq0, \label{app-WDPt1}\\
&&G(w^k)^T\lambda^k=0, ~~\lambda^k\geq0,\nonumber
\end{eqnarray}
and the gradients
\begin{eqnarray}
\{\nabla G_i(w^k): 
\lambda_i^k>0, 0\le i\le 2p\}\cup\{\nabla H_i(w^k): \mu_i^k\not=0\} \label{independent}
\end{eqnarray}
are linearly independent. To show that $\bar{w}$ is a KKT point of \eqref{app-WDP}, it is sufficient to show the boundedness of the sequence of multipliers $\{\lambda_0^k,\lambda^k,\mu^k\}$.

Assume by contradiction that $\{\lambda_0^k,\lambda^k,\mu^k\}$ is unbounded. Taking a subsequence if necessary, we may assume
\begin{eqnarray}\label{0}
\lim_{k\to\infty}\frac{(\lambda_0^k,\lambda^k,\mu^k)}
{\|(\lambda_0^k,\lambda^k,\mu^k)\|} = (\lambda_0^*,\lambda^*,\mu^*). \end{eqnarray}
In particular, for every $k$ sufficiently large, we have
\begin{eqnarray}
\lambda_i^*>0 \ \Rightarrow \ \lambda_i^k>0, \quad \mu_i^*\not=0 \ \Rightarrow \ \mu_i^k\not= 0. \label{supp}
\end{eqnarray}
Dividing by $\|(\lambda_0^k,\lambda^k,\mu^k)\|$ in \eqref{app-WDPt1} and taking a limit, we have
\begin{eqnarray}\label{app-WDPt2}
&&\lambda_0^*\nabla G_0(\bar{w})+\sum_{\lambda_i^*>0}\lambda_i^*\nabla G_i(\bar{w})+\sum_{\mu_i^*\not=0}\mu_i^*\nabla H_i(\bar{w})=0, ~~\lambda_0^*\geq0.
\end{eqnarray}
Since $\|(\lambda_0^*,\lambda^*,\mu^*)\|=1$ by \eqref{0}, \eqref{app-WDPt2} implies that the gradients
\begin{eqnarray*}
\{\nabla G_i(\bar{w}): \lambda_i^*>0, 0\le i\le 2p\}\cup\{\nabla H_i(\bar{w}): \mu_i^*\not=0\}
\end{eqnarray*}
are positive-linearly dependent. Since problem \eqref{app-WDP} satisfies the CPLD at $\bar{w}$, the gradients
\begin{eqnarray*}
\{\nabla G_i(w): \lambda_i^*>0, 0\le i\le 2p\}\cup\{\nabla H_i(w): \mu_i^*\not=0\}
\end{eqnarray*}
are linearly dependent in a neighborhood. This, together with \eqref{supp}, contradicts the linear independence of the vectors in \eqref{independent}.
Consequently, $\{\lambda_0^k,\lambda^k,\mu^k\}$ is bounded.
This completes the proof.
\end{proof}

%%%%%%%%%%%%%%%%%%%%%%%%%%%%%%%%
%%%%%%%%%%%%%%%%%%%%%%%%%%%%%%%%
%%%%%%%%%%%%%%%%%%%%%%%%%%%%%%%%
\section{Numerical experiments}
\label{numerical}

In this section, we report our numerical experience in applying the proposed approach to solve linear bilevel programs. In our experiments, we generated  tested problems randomly and  made a numerical comparison of the (WDP) approach with the popular MPEC approach by solving them respectively as nonlinear optimization problems and through relaxation methods. In particular, all experiments were implemented in MATLAB 9.8.0 and run on the same computer with Windows Server 2019 Standard, Intel(R) Xeon(R) Platinum 8268, CPU $@$ 2.90 GHz 2.89 GHz (4 processors) and 2.00 TB memory. 

%%%%%%%%%%%%%%%%%%%%%%%%%%%%%%%%
%%%%%%%%%%%%%%%%%%%%%%%%%%%%%%%%
%%%%%%%%%%%%%%%%%%%%%%%%%%%%%%%%
\subsection{Test problems}

We designed a procedure to generate the following linear bilevel program:
\begin{eqnarray}
\min &&c_1^{T} x+c_2^{T} y\nonumber\\
\mbox{s.t.}&&A_1x\leq b_1,\label{linear-linear BP}\\
&&y\in \arg\min\limits_{y}\{d_2^{T} y: A_2x+B_2y \leq b_2,\ l_b\leq y\leq u_b\},\nonumber
\end{eqnarray}
where all coefficients $A_1\in \mathbb{R}^{p\times n}$, $A_2\in \mathbb{R}^{q\times n}$, $B_2\in \mathbb{R}^{q\times m}$, $b_1\in \mathbb{R}^p$, $b_2\in \mathbb{R}^q$, $c_1\in \mathbb{R}^n$, and $\{c_2,d_2\}\subset \mathbb{R}^m$ were generated sparsely by \emph{sprand} in $[0,1]$ with density to be 0.1, the lower and upper bounds were given by $l_b=-10*ones(m,1)$ and $u_b=10*ones(m,1)$ respectively. Since the lower-level program is linear, \eqref{linear-linear BP} can be equivalently transformed into the MPEC
\begin{eqnarray}
\min &&c_1^{T} x+c_2^{T} y \nonumber\\
\mbox{s.t.}&&A_1x \leq b_1, ~d_2+B_2^{T} u_1+u_2-u_3=0,\label{linear-linear MPEC}\\
&&0\leq u_1 \perp A_2x+B_2y- b_2\leq 0,\nonumber\\
&&0\leq u_2 \perp y-u_b\leq 0, ~0\leq u_3 \perp y-l_b\geq 0\nonumber
\end{eqnarray}
or (WDP)
\begin{eqnarray*}
\min &&c_1^{T} x+c_2^{T} y \nonumber\\
\mbox{s.t.}&&A_1x \leq b_1, ~d_2+B_2^{T} u_1+u_2-u_3=0,\label{linear-linear WDP} \\
&&d_2^{T} y\leq d_2^{T} z+u_1^{T}(A_2x+B_2z-b_2)+u_2^{T}(z-u_b)+u_3^{T}(l_b-z), \nonumber\\
&&u_1\geq0,~u_2\geq0,~u_3\geq0, ~A_2x+B_2y \leq b_2,~l_b\leq y\leq u_b.\nonumber
\end{eqnarray*}

In our experiments, we generated 20 problems, which contain two groups: Problems 1-10 were generated with $\{n=50, p=40, m=60, q=50\}$ and problems 11-20 were generated with $\{n=50, p=40, m=100, q=80\}$. We only adjusted the numbers $m$ and $q$ because the lower-level variables and constraints are the essential  factors in solving bilevel programs.

Besides, when evaluating the numerical results, we mainly focused on three aspects: the feasibility of the approximation solutions, the objective values, and the CPU times. In particular, 
we adopted the following unified criterion to measure the infeasibility for both MPEC and (WDP) in our tests: 
\begin{eqnarray*}
{\rm{Infeasibility}}&:=&\|\max(0,A_1x^k-b_1)\|+\|\max(0, A_2x^k+B_2y^k-b_2)\|\\
&&+\|\max(0,y^k-u_b)\|+\|\max(0,l_b-y^k)\|+|d_2^{T}y^k-h^*(x^k)|,
\end{eqnarray*}
where $(x^k,y^k,u^k)$ or $(x^k,y^k,z^k,u^k)$ is the current approximation solution to the MPEC or (WDP), and $h^*(x^k):=\min \{d_2^{T}y: A_2x^k+B_2y\leq b_2,~l_b\leq y\leq u_b\}$. The last term in the definition of Infeasibility is used to approximate the distance from $y^k$ to the solution set of the lower-level program, which is  difficult to calculate in practice.

%%%%%%%%%%%%%%%%%%%%%%%%%%%%%%%%
%%%%%%%%%%%%%%%%%%%%%%%%%%%%%%%%
%%%%%%%%%%%%%%%%%%%%%%%%%%%%%%%%
\subsection{Solving WDP and MPEC as nonlinear programs }\label{directsolver}

\doublerulesep=0.4pt
\begin{table}
\caption{Numerical results by solving WDP and MPEC as nonlinear programs}
 \label{tab:as nonlinear (n=50; m=60; p=40; q=50)}%
 \begin{tabular}{cccccc}
\hline\hline\hline\noalign{\smallskip}{Problem}
& {Method} & {Upper-Level Objective Value}  & {Infeasibility}  & {CPU Time}\\

\noalign{\smallskip}\hline\hline\hline\noalign{\smallskip}{1}
& MPCC  & \textbf{-3.1463e+12} & 2.1900e-11          & 1.4554e+02 \\

& WDP  & -3.3923e+08                  & \textbf{5.6843e-14} & \textbf{6.4458e+01} \\

\noalign{\smallskip}\hline\noalign{\smallskip}{2}
& MPCC  & -2.5857e+09                   & \textbf{4.8005e-07} & \textbf{7.7537e+01} \\

& WDP  & \textbf{-3.0284e+13}  & 1.4604e+09          & 1.0761e+02 \\

\noalign{\smallskip}\hline\noalign{\smallskip}{3}
& MPCC  & -3.7596e+12                   & \textbf{9.0261e-05} & 2.4301e+02 \\

& WDP  & \textbf{-2.2363e+22}  & 3.2708e+10          & \textbf{2.3440e+02} \\

\noalign{\smallskip}\hline\noalign{\smallskip}{4}
& MPCC  & -1.0616e+12                    & \textbf{1.3316e+05} & \textbf{1.0887e+02} \\

& WDP  & \textbf{-2.9227e+21}  & 1.6209e+07          & 1.6791e+02 \\

\noalign{\smallskip}\hline\noalign{\smallskip}{5}
& MPCC  & -1.1837e+10                    & \textbf{1.7053e-13} & 4.9124e+02 \\

& WDP  & \textbf{-6.3853e+17}  & 8.1583e+05          & 9.1536e+01 \\

\noalign{\smallskip}\hline\noalign{\smallskip}{6}
& MPCC  & -2.5239e+05                    & 6.2759e+01          & 1.1706e+02 \\

& WDP  & \textbf{-1.5364e+08}  & \textbf{4.0963e-08} & \textbf{6.3872e+01} \\

\noalign{\smallskip}\hline\noalign{\smallskip}{7}
& MPCC  & \textbf{-1.0972e+10}           & 2.6241e-06          & 1.1103e+02 \\

& WDP  & -2.3061e+09          & \textbf{3.0972e-08} & \textbf{3.7846e+01} \\

\noalign{\smallskip}\hline\noalign{\smallskip}{8}
& MPCC  & \textbf{-1.1108e+10}  & 1.8574e-04          & \textbf{1.6481e+02} \\

& WDP  & -6.3271e+06                    & \textbf{7.1054e-15} & 4.6037e+02 \\

\noalign{\smallskip}\hline\noalign{\smallskip}{9}
& MPCC  & -3.5679e+11                    & 4.4173e+02          & 1.3203e+02 \\

& WDP  & \textbf{-3.8552e+12}  & \textbf{2.3625e-04} & \textbf{9.3630e+01} \\

\noalign{\smallskip}\hline\noalign{\smallskip}{10}
& MPCC  & -1.6421e+11                   & 5.1916e+04          & \textbf{9.1017e+01} \\

& WDP  & \textbf{-3.1135e+13}  & \textbf{3.1974e-04} & 1.2583e+02 \\

\noalign{\smallskip}\hline\hline\hline\noalign{\smallskip}{11}
& MPCC  & -8.8731e+09                    & \textbf{4.1478e-06} & \textbf{3.5987e+02} \\

& WDP  & \textbf{-4.9310e+12}  & 4.4921e-06          & 3.6791e+02 \\

\noalign{\smallskip}\hline\noalign{\smallskip}{12}
& MPCC  & -5.9250e+00            & 6.1676e+02          & \textbf{5.8209e+01} \\

& WDP  & \textbf{-1.3248e+10}            & \textbf{3.1468e+02} & 7.7642e+02 \\

\noalign{\smallskip}\hline\noalign{\smallskip}{13}
& MPCC  & -3.7345e+10                    & \textbf{1.8982e-06} & 6.2414e+02 \\

& WDP  & \textbf{-7.3841e+10}  & 2.2366e+02          & \textbf{4.4411e+02} \\

\noalign{\smallskip}\hline\noalign{\smallskip}{14}
& MPCC  & -4.9281e+00           & 3.3777e+02          & \textbf{6.2954e+01} \\

& WDP  & \textbf{-6.5233e+12}            & \textbf{1.4604e-01} & 3.2927e+02 \\

\noalign{\smallskip}\hline\noalign{\smallskip}{15}
& MPCC  & -6.4304e+08                    & \textbf{3.5476e-08} & 6.7769e+02 \\

& WDP  & \textbf{-3.3281e+10}  & 1.2171e+01          & \textbf{2.8304e+02} \\

\noalign{\smallskip}\hline\noalign{\smallskip}{16}
& MPCC  & -1.2443e+01                    & 3.5345e+02          & \textbf{1.1613e+02} \\

& WDP  & \textbf{-9.6905e+12}  & \textbf{1.0489e+02} & 3.1826e+02 \\

\noalign{\smallskip}\hline\noalign{\smallskip}{17}
& MPCC  & -7.2061e+09           & \textbf{8.1419e+00} & 8.5065e+02 \\

& WDP  & \textbf{-2.5564e+11}           & 5.9209e+02          & \textbf{4.0125e+02} \\

\noalign{\smallskip}\hline\noalign{\smallskip}{18}
& MPCC  & \textbf{-1.0965e+08}           & 7.0720e-08          & 3.1691e+02 \\

& WDP  & -4.3347e+07           & \textbf{3.3506e-08} & \textbf{1.4159e+02} \\

\noalign{\smallskip}\hline\noalign{\smallskip}{19}
& MPCC  & -4.2493e+00           & 5.9103e+02          & \textbf{5.2648e+01} \\

& WDP  & \textbf{-7.2757e+15}           & \textbf{1.9079e+02} & 7.5048e+02 \\

\noalign{\smallskip}\hline\noalign{\smallskip}{20}
& MPCC  & -2.1158e+08                    & \textbf{2.0894e-07} & 4.7101e+02 \\

& WDP  & \textbf{-5.2877e+10}  & 8.5212e-06          & \textbf{9.7612e+01} \\

\noalign{\smallskip}\hline\hline\hline
    \end{tabular}
\footnotesize{\\${}^a$ Problems 1-10 were generated with $n=50, p=40; m=60, q=50$.}
\footnotesize{\\${}^b$ Problems 11-20 were generated with $n=50, p=40; m=100, q=80$.}
\end{table}

In the first part of our experiments, we solved both MPEC and (WDP) directly by invoking the SQP algorithm of the solver \emph{fmincon}. We chose all starting points to be the origins and selected the required parameters, together with the termination criterion, to default settings as set in the package. See Table \ref{tab:as nonlinear (n=50; m=60; p=40; q=50)} for the numerical results on the tested problems.

There are two aspects different from what we had predicted before the experiments. The first is related to the feasibility of the approximation solutions. Since all lower-level programs are equivalent to their KKT conditions for the tested problems, the MPEC approach should have greater advantages in feasibility than (WDP), but the numerical results indicate that this is not always the case, at least in our tests. In fact, from Table \ref{tab:as nonlinear (n=50; m=60; p=40; q=50)}, we can see that, if we set the tolerance for infeasibility to be $10^{-4}$, there were 10 and 8 satisfactory cases among 20 problems for MPEC and (WDP) respectively.

The second aspect is that the efficiency of solving linear (BP) through the (WDP) approach directly was lower than we had expected before the experiments. From Table \ref{tab:as nonlinear (n=50; m=60; p=40; q=50)}, among the five cases where the infeasibility of both was lower than $10^{-4}$ (that is, problems 1, 7, 11, 18, 20), there were only two cases where the (WDP) approach had advantages over the MPEC approach in the objective values. This may be caused by the increase of dimension of (WDP). On the other hand, it reminds us that (WDP) is still difficult to solve directly in many cases, although it may satisfy the MFCQ at its feasible points.

Delightfully, we observed that, among the five cases mentioned above, the (WDP) approach had greater advantages than the MPEC approach in the CPU times. Our experiments on relaxation methods given in the next subsection show that the (WDP) approach had overall advantages over the MPEC approach in both objective values and CPU times in our tests.

%%%%%%%%%%%%%%%%%%%%%%%%%%%%%%%%
%%%%%%%%%%%%%%%%%%%%%%%%%%%%%%%%
%%%%%%%%%%%%%%%%%%%%%%%%%%%%%%%%
\subsection{Solving (WDP) and MPEC by relaxation methods}

\begin{table}
  \caption{Numerical results for WDP and MPEC by relaxation methods}
\label{tab:relaxation method (n=50; m=60; p=40; q=50)}%

\begin{tabular}{cccccc}
\hline\hline\hline\noalign{\smallskip}{Problem}
& {Method} & {Upper-Level Objective Value}  & {Infeasibility}  & {CPU Time}\\

\noalign{\smallskip}\hline\hline\hline\noalign{\smallskip}{1}
& MPCC  & -3.6778e+22                   & 5.6843e-14          & \textbf{4.8954e+02} \\

& WDP  & \textbf{-1.5743e+24}  & 5.6843e-14          & 1.6175e+03          \\

\noalign{\smallskip}\hline\noalign{\smallskip}{2}
& MPCC  & -2.1392e+11                    & 0.0000e+00          & 9.5534e+02         \\

& WDP  & \textbf{-1.0887e+18}  & 0.0000e+00          & \textbf{8.8564e+02} \\

\noalign{\smallskip}\hline\noalign{\smallskip}{3}
& MPCC  & \textbf{-2.9363e+11}  & 6.8764e-08          & 1.8591e+02           \\

& WDP  & -5.8726e+10                   & \textbf{2.8422e-14} & \textbf{1.0515e+02}  \\

\noalign{\smallskip}\hline\noalign{\smallskip}{4}
& MPCC  & -2.5200e+08           & \textbf{0.0000e+00} & \textbf{1.8158e+02} \\

& WDP  & \textbf{-1.7895e+10}           & 2.7830e-11          & 1.1953e+03        \\

\noalign{\smallskip}\hline\noalign{\smallskip}{5}
& MPCC  & -5.2419e+10                   & 1.1377e-11          & 1.2217e+03         \\

& WDP  & \textbf{-2.6210e+11}  & \textbf{3.3307e-16} & \textbf{3.9862e+02}  \\

\noalign{\smallskip}\hline\noalign{\smallskip}{6}
& MPCC  & -3.3349e+10                   & 1.6620e-11          & \textbf{1.0940e+02}  \\
& WDP  & \textbf{-1.7702e+15}  & \textbf{5.6843e-14} & 1.1025e+02           \\

\noalign{\smallskip}\hline\noalign{\smallskip}{7}
& MPCC  & -1.1298e+10                  & 5.6843e-14          & 9.9915e+02          \\

& WDP  & \textbf{-3.5136e+16}  & 5.6843e-14          & \textbf{8.3334e+02} \\

\noalign{\smallskip}\hline\noalign{\smallskip}{8}
& MPCC  & \textbf{-1.4392e+16} & \textbf{0.0000e+00} & 3.4346e+03          \\

& WDP  & -3.3058e+10                 & 1.1559e-11          & \textbf{1.4444e+03} \\

\noalign{\smallskip}\hline\noalign{\smallskip}{9}
& MPCC  & -7.0267e+09                   & 6.8374e-07          & 1.0926e+03        \\

& WDP  & \textbf{-2.3466e+15} & \textbf{5.6843e-14} & \textbf{1.1597e+02} \\

\noalign{\smallskip}\hline\noalign{\smallskip}{10}
& MPCC  & \textbf{-3.4852e+17} & 2.8422e-14          & 3.0776e+03         \\

& WDP  & -2.3240e+17                   & 2.8422e-14          & \textbf{8.2931e+01}  \\

\noalign{\smallskip}\hline\hline\hline\noalign{\smallskip}{11}
& MPCC  & -5.5652e+10                   & 1.8271e-12          & 1.2597e+03        \\

& WDP  & \textbf{-1.2063e+11}  & \textbf{0.0000e+00} & \textbf{7.7513e+02}\\

\noalign{\smallskip}\hline\noalign{\smallskip}{12}
& MPCC  & -6.6887e+10          & 9.2383e-07          & \textbf{1.4390e+03} \\

& WDP  & \textbf{-2.7658e+11}           & \textbf{2.2737e-13} & 2.7867e+03          \\

\noalign{\smallskip}\hline\noalign{\smallskip}{13}
& MPCC  & -1.0222e+10                   & 9.9172e-07          & 2.0430e+03         \\

& WDP  & \textbf{-5.1128e+10} & \textbf{1.3984e-12} & \textbf{7.0663e+02}  \\

\noalign{\smallskip}\hline\noalign{\smallskip}{14}
& MPCC  & -1.6065e+10                   & 1.7830e-12          & \textbf{2.3482e+03} \\
& WDP  & \textbf{-7.3393e+11}  & \textbf{5.6843e-14} & 2.5877e+03            \\

\noalign{\smallskip}\hline\noalign{\smallskip}{15}
& MPCC  & -1.5798e+11                   & \textbf{0.0000e+00} & 7.3459e+03           \\

& WDP  & \textbf{-3.9495e+12}  & 2.1346e-32          & \textbf{5.8045e+02}  \\

\noalign{\smallskip}\hline\noalign{\smallskip}{16}
& MPCC  & \textbf{-1.0787e+11}          & 1.2280e-10          & 1.9163e+03         \\

& WDP  & -3.3889e+08          & \textbf{5.6843e-14} & \textbf{3.1293e+02}  \\

\noalign{\smallskip}\hline\noalign{\smallskip}{17}
& MPCC  & \textbf{-2.5465e+19}  & \textbf{0.0000e+00} & \textbf{2.4133e+03}  \\
& WDP  & -6.3969e+10                   & 4.7536e-07          & 6.6342e+03         \\

\noalign{\smallskip}\hline\noalign{\smallskip}{18}
& MPCC  & -2.7215e+10                   & 7.4023e-13          & \textbf{1.0163e+03}  \\

& WDP  & \textbf{-3.2015e+11} & \textbf{5.6843e-14} & 2.1804e+03         \\

\noalign{\smallskip}\hline\noalign{\smallskip}{19}
& MPCC  & \textbf{-1.4700e+10}           & 9.2752e-13          & 1.3563e+03          \\

& WDP  & -5.6359e+07          & \textbf{1.1369e-13} & \textbf{9.1482e+02} \\

\noalign{\smallskip}\hline\noalign{\smallskip}{20}
& MPCC  & \textbf{-7.6390e+10}  & 4.2411e-07          & \textbf{4.5808e+02} \\

& WDP  & -1.5278e+10                    & \textbf{3.1483e-09} & 2.2327e+03        \\

\noalign{\smallskip}\hline\hline\hline
    \end{tabular}

\end{table}

In this subsection, we report our numerical results in solving (WDP) and MPEC by relaxation methods. The relaxation method for (WDP) has been studied in Section 5, while the relaxation scheme for MPEC employed here follows from the work \cite{Scholtes2001}, by using 
\begin{eqnarray*}
\mathrm{MPEC}(t)\quad\min &&c_1^{T} x+c_2^{T} y \nonumber\\
\mbox{s.t.}&&A_1x \leq b_1,~d_2+B_2^{T} u_1+u_2-u_3=0,\nonumber\\
&&u_1\geq0,~u_2\geq0,~u_3\geq0,~A_2x+B_2y \leq b_2,~l_b\leq y\leq u_b,\nonumber\\
&&u_1^{T}(A_2x+B_2y-b_2)+u_2^{T}(y-u_b)+u_3^{T}(l_b-y)\leq t\nonumber
\end{eqnarray*}
to approximate the MPEC \eqref{linear-linear MPEC}, 
where $t>0$ is an approximation parameter.

We implemented Algorithm \ref{alg:WDP-Relaxed} given below to solve the tested problems. Note that, for the linear bilevel programs, the feasibility of $\tilde{w}^k$ to (WDP) or MPEC is evident.

\begin{algorithm}
  \caption{}
  \label{alg:WDP-Relaxed}

\textbf{Step 0} Choose $\tilde{x}^0$ such that $A_1\tilde{x}^0 \leq b_1$, an initial parameter $t_0>0$, a parameter $\sigma\in(0,1)$, and some accuracy tolerances $\{\epsilon_{\rm p},\epsilon_{\rm r},\delta_{\min}\}$. Set $k:=0$.

\textbf{Step 1} Solve the lower-level program in \eqref{linear-linear BP} with $x=\tilde{x}^k$ by \emph{linprog} to get $\tilde{y}^k$ and $\tilde{u}^k$. Set $\tilde{w}^k:=(\tilde{x}^k,\tilde{y}^k,\tilde{y}^k,\tilde{u}^k)$ in the case of (WDP) and $\tilde{w}^k:=(\tilde{x}^k,\tilde{y}^k,\tilde{u}^k)$ in the case of MPEC. If $\tilde{w}^k$ satisfies a termination criterion with $\epsilon_{\rm p}$ as the accuracy tolerance, stop. Otherwise, go to Step 2.

\textbf{Step 2} Solve WDP$(t_k)$ or MPEC$(t_k)$ by \emph{fmincon} with $\tilde{w}^k$ as the starting point and $\epsilon_{\rm r}$ as the accuracy tolerance to obtain an iterative $w^k:=(x^k,y^k,z^k,u^k)$ or $w^k:=(x^k,y^k,u^k)$. If $w^k$ satisfies the above termination criterion, 
stop. Otherwise, go to Step 3.

\textbf{Step 3}
Set $\tilde{x}^{k+1}:=\tilde{x}^k$ and $t_{k+1}:=\max\{\sigma t_k, \delta_{\min}\}$. Let $k:=k+1$ and go to Step 1.
\end{algorithm}

In our experiments, we chose $\tilde{x}^0$ to be the origins and set $t_0=1$, $\sigma=10^{-1}$, $\epsilon_{\rm p}=10^{-8}$, and $\epsilon_{\rm r}=\delta_{\min}=10^{-16}$. The termination criteria for both (WDP) and MPEC were defined as the residuals of some kinds of optimality conditions. See Table \ref{tab:relaxation method (n=50; m=60; p=40; q=50)} for the numerical results on the tested problems.

From Table \ref{tab:relaxation method (n=50; m=60; p=40; q=50)}, we can see that there were 18 and 15 cases with sufficient feasibility for the (WDP) and MPEC approaches respectively. Among the 14 cases with sufficient feasibility for both, the (WDP) approach had overall advantages over the MPEC approach in both objective values and CPU times. In particular, compared with Table \ref{tab:as nonlinear (n=50; m=60; p=40; q=50)}, the CPU times did not increase significantly. This indicates that, at least in our tests, the relaxation method based on the (WDP) reformulation was quite efficient.

%%%%%%%%%%%%%%%%%%%%%%%%%%%%%%%%
%%%%%%%%%%%%%%%%%%%%%%%%%%%%%%%%
%%%%%%%%%%%%%%%%%%%%%%%%%%%%%%%%
\section{Conclusions}
\label{conclu}

We have shown that, under mild conditions such as pseudoconvexity, the bilevel program (BP) is equivalent to the new single-level optimization problem (WDP) in the globally or locally optimal sense. Unlike the existing single-level reformulations (SP), (VP), and (LDP), (WDP) does not involve any function without analytic expression. Compared with another existing single-level reformulation MPEC, we have shown by an example that (WDP) may satisfy the MFCQ at its feasible points, unlike the MPEC that fails to satisfy the MFCQ at any feasible point, and hence (WDP) is relatively easy to deal with. Based on our realization that (WDP) is still difficult to solve directly in some cases, we have proposed a relaxation method for solving (WDP) and established a convergence analysis. Comprehensive numerical experiments indicate that, although solving (WDP) directly did not perform very well in our tests, the relaxation method based on the (WDP) reformulation was quite efficient. As future work, we will be devoted to analyzing the (WDP) reformulation more deeply from theoretical and numerical aspects so as to develop more efficient algorithms for the bilevel program (BP).

%%%%%%%%%%%%%%%%%%%%%%%%%%%%%%%%
%%%%%%%%%%%%%%%%%%%%%%%%%%%%%%%%
%%%%%%%%%%%%%%%%%%%%%%%%%%%%%%%%

\end{document}